\documentclass[a4paper,11pt]{amsart}

\usepackage{amssymb,amsmath,amsthm,mathrsfs,enumerate,graphicx, color}
\usepackage{tikz}
\usepackage[pdfpagelabels,colorlinks,linkcolor=blue,citecolor=black,urlcolor=blue]{hyperref}
\usepackage{esint}
\newtheorem{thm}{Theorem}[section]

\newtheorem{lem}[thm]{Lemma}
\newtheorem{prop}[thm]{Proposition}
\newtheorem{defn}[thm]{Definition}
\newtheorem{rem}[thm]{Remark}

\newcommand{\SLvc}{\mathcal{S}_{\vc}}

\newcommand{\SR}{\mathscr{S}(\mathbb{R})}

\newcommand{\lesi}{\lesssim}
\newcommand{\x}{\mathbb{\times}}

\newcommand{\B}{\dot{B}}
\newcommand{\F}{\dot{F}}

\newcommand{\sL}{\sqrt{L}}
\newcommand{\dx}{d\mu(x)}
\newcommand{\dy}{d\mu(y)}
\newcommand{\dz}{d\mu(z)}

\newcommand{\supp}{\operatorname{supp}}
\newcommand{\f}{\frac}

\newcommand{\Hk}{\mathcal{H}}
\newcommand{\tx}{\tau^x}
\newcommand{\ty}{\tau^y}
\newcommand{\al}{\alpha}
\newcommand{\p}{\partial}
\newcommand{\si}{\sigma}

\newcommand{\su}{\subset}
\newcommand{\vc}{\infty}

\title[Spectral multipliers on new Besov and Triebel--Lizorkin spaces]{Spectral multipliers of self-adjoint operators on  Besov and Triebel--Lizorkin spaces
associated to operators}         

\author{The Anh Bui}
\address{Department of Mathematics, Macquarie University, NSW 2109,
Australia}
\email{the.bui@mq.edu.au, bt\_anh80@yahoo.com}
 
\author{Xuan Thinh Duong}
\address{Department of Mathematics, Macquarie University, NSW 2109,
Australia}

\email{xuan.duong@mq.edu.au}

\subjclass[2010]{42B30, 42B35, 47B38}

\thanks{{\it Key words and phrases}: spectral multiplier, Besov spaces, Triebel--Lizorkin spaces}

\begin{document}

\begin{abstract}
	Let $X$ be a space of homogeneous type and let $L$ be a nonnegative self-adjoint operator on $L^2(X)$ 
	which satisfies a Gaussian estimate on its heat kernel. In this paper we prove a H\"omander type spectral multiplier theorem 
	for $L$ on the Besov and Triebel--Lizorkin spaces associated to $L$. Our work not only recovers the boundedness of the 
	spectral multipliers on $L^p$ spaces and Hardy spaces associated to $L$, but also is the first one which proves the boundedness 
	of a general spectral theorem on Besov and Triebel--Lizorkin spaces.   
\end{abstract}
\date{}

\maketitle

\tableofcontents

\section{Introduction}
Let $X$ be a space of homogeneous type, with quasi distance $d$ and
$\mu$ is a nonnegative Borel measure on $X$, which satisfies the doubling property (\ref{doubling1}) below. \emph{In this paper, we assume that $\mu(X)=\infty.$}\\

For $x\in X$ and $r>0$ we set $B(x, r)=\{y\in X:d(x,y)<r\}$ to be the open ball of radius $r >0$ and centered at $x\in
X$, and $V(x,r)=\mu(B(x, r))$. The doubling
property of $\mu$ provides that there exists a constant $C>0$ so
that
\begin{equation}\label{doubling1}
V(x,2r)\leq CV(x,r)
\end{equation}
for all $x\in X$ and $r>0$.\\
The doubling property (\ref{doubling1}) yields that there exists $n>0$ so that
\begin{equation}\label{doubling2}
V(x,\lambda r)\leq C\lambda^nV(x,r),
\end{equation}
for some positive constant $n$ uniformly for all $\lambda\geq 1,
x\in X$ and $r>0$; and that
\begin{equation}\label{doubling3}
V(x,r)\leq C\Big(1+\frac{d(x,y)}{r}\Big)^{\tilde n}V(y,r),
\end{equation}
uniformly for all $x,y\in X$, $r>0$ and for some $\tilde n \in [0,n]$.\\

Let $L$ be a non-negative self-adjoint  operator on $L^2(X)$ which generates the semigroup $\{e^{-tL}\}_{t>0}$. Denote by $p_t(x,y)$ the  kernel of the semigroup $e^{-tL}$. In this paper, we assume that the kernel $p_t(x,y)$ satisfies a Gaussian upper bound, i.e., there exist positive constants $C$ and  $c$ so that for all $x,y\in X$ and $t>0$,
\begin{equation}
\tag{GE}\label{GE}
\displaystyle |p_t(x,y)|\leq \f{C}{\mu(B(x,\sqrt{t}))}\exp\Big(-\f{d(x,y)^2}{ct}\Big).
\end{equation}

Denote by $E_L(\lambda)$ a spectral resolution of $L$. Then by spectral theory, for any bounded Borel function $F:[0,\vc)\to \mathbb{C}$ we can define
$$
F(L)=\int_0^\vc F(\lambda)dE_L(\lambda)
$$
as a bounded operator on $L^2(X)$. It is natural to raise a question on the boundedness of the spectral multipliers $F(L)$ on various function spaces under some suitable smoothness conditions on $F$. We note that the problem on the boundedness of the spectral multipliers has had a long history and has been received a great deal of attention by many mathematicians. The early result for the $L^p$ boundedness for the spectral multiplier in the standard case when $L=-\Delta$ is Laplace operator on the Euclidean space $\mathbb{R}^n$ was obtained by L. H\"ormander \cite{Ho}. Then this result has been extended to various settings such as Lie groups of polynomial growth, nilpotent groups and spaces of
homogeneous type. See for example \cite{Alex, MM, C2, MS, Heb, DOS, DSY} and the references therein. The $L^p$-boundedness of the spectral multipliers for a general operator $L$ satisfying the Gaussian upper bounds was obtained in \cite{DOS}. Then the authors in \cite{DSY} extended the result in \cite{DOS} to the weighted $L^p$-estimates. The work \cite{DSY} can be viewed as an extension of the classical result for the spectral multipliers of the standard Laplacian in \cite{KW}. General spectral multiplier theorems on Hardy spaces 
associated to operators were obtained in \cite{DY, B}. 

The main aim of this paper is to prove the boundedness of the spectral multipliers $F(L)$ on new Besov and Triebel--Lizorkin spaces associated to the operator $L$. More precisely, we are able to prove the following result:

\begin{thm}\label{mainthm-spectralmultipliers-Miklin} 	Let $s>\f{n}{2}$. Then for any bounded Borel function $F$ such that $\sup_{t>0}\|\eta\, \delta_tF\|_{W^{\vc}_s}<\vc$ where $\delta_tF(\cdot)= F(t\cdot)$ and $\eta$ is a $C^\vc_c(\mathbb R_+)$ function, not identically zero, we have:
	\begin{enumerate}[\rm (a)]
		\item the spectral multiplier $F(\sqrt L)$ is bounded on $\F^{\alpha, L}_{p,q}(X)$ provided that $\alpha\in \mathbb{R}$, $0<p,q <\vc$ and $s>n(\f{1}{1\wedge p\wedge  q}-\f{1}{2})$, i.e.,
		\[
		\|F(\sqrt L)\|_{\F^{\alpha, L}_{p,q}(X)\to \F^{\alpha, L}_{p,q}(X)}\lesi F(0)+\sup_{t>0}\|\eta\, \delta_tF\|_{W^{\vc}_s}.
		\]
		\item the spectral multiplier $F(\sqrt L)$ is bounded on $\B^{\alpha, L}_{p,q}(X)$ provided that $\alpha\in \mathbb{R}$, $0<p<\vc$, $0<q\le \vc$ and $s>n(\f{1}{1\wedge p\wedge  q}-\f{1}{2})$, i.e.,
		$$
		\|F(\sqrt L)\|_{\B^{\alpha, L}_{p,q}(X)\to \B^{\alpha, L}_{p,q}(X)}\lesi F(0) +\sup_{t>0}\|\eta\, \delta_tF\|_{W^{\vc}_s}.
		$$		
	\end{enumerate}
\end{thm}
Note that by using a different approach the authors in \cite{G.etal2} proved similar estimates to those in Theorem \ref{mainthm-spectralmultipliers-Miklin}  under the stronger assumptions  that $s>\f{n}{1\wedge p\wedge q} + \f{n}{2}$ and $L$ satisfies two additional conditions (H) and (C) (See Remark \ref{rem1}).

In fact, the condition $\sup_{t>0}\|\eta\, \delta_tF\|_{W^{\vc}_s}<\vc$ in Theorem \ref{mainthm-spectralmultipliers-Miklin} can be improved in the following spectral multiplier theorem of H\"ormander type.

\begin{thm}\label{mainthm-spectralmultipliers} 	Let $s>\f{n}{2}$ and let   $\alpha\in \mathbb{R}$ and $0<p,q<\vc$.
		Assume that for any $R>0$ and all Borel functions $F$  such that\, {\rm supp}\, $F\subseteq [0, R]$, the following holds for some $\tilde q\in [2, \infty]$:
	\begin{equation}\label{eq1-mainthmsm}
	\int_X |K_{F(\sqrt{L})}(x,y)|^2 d\mu(x) \leq \f{C}{V(y, R^{-1})} \|\delta_R F\|^2_{\tilde q}.
	\end{equation}
	
	Assume that $F$ is a  bounded Borel function satisfying the following condition
	\begin{equation}
	\label{smoothness condition}
	\sup_{t>0}\|\eta\, \delta_tF\|_{W^{\tilde q}_s}<\infty
	\end{equation} 
	where $\delta_tF(\cdot)= F(t\cdot)$ and $\eta$ is a $C^\vc_c(\mathbb R_+)$ function, not identically zero.

	Then we have:
	\begin{enumerate}[\rm (a)]
		\item the spectral multiplier $F(\sqrt L)$ is bounded on $\F^{\alpha, L}_{p,q}(X)$ provided that $\alpha\in \mathbb{R}$, $0<p,q <\vc$ and $s>\max\left\{n(\f{1}{1\wedge p\wedge  q}-\f{1}{2}), \f{1}{\tilde q}\right\}$, i.e.,
		\[
		\|F(\sqrt L)\|_{\F^{\alpha, L}_{p,q}(X)\to \F^{\alpha, L}_{p,q}(X)}\lesi F(0)+\sup_{t>0}\|\eta\, \delta_tF\|_{W^{\tilde q}_s}.
		\]
		\item the spectral multiplier $F(\sqrt L)$ is bounded on $\B^{\alpha, L}_{p,q}(X)$ provided that $\alpha\in \mathbb{R}$, $0<p<\vc$, $0<q\le \vc$ and $s>\max\left\{n(\f{1}{1\wedge p\wedge  q}-\f{1}{2}),\f{1}{\tilde q}\right\}$, i.e.,
		$$
				\|F(\sqrt L)\|_{\B^{\alpha, L}_{p,q}(X)\to \B^{\alpha, L}_{p,q}(X)}\lesi F(0) +\sup_{t>0}\|\eta\, \delta_tF\|_{W^{\tilde q}_s}.
		$$		
	\end{enumerate}
\end{thm}

We would like to emphasize that the sharp spectral multiplier theorem on Besov and Triebel--Lizorkin spaces was first obtained in \cite{Tri} for the classical case when $L=-\Delta$ is the Laplacian on $\mathbb{R}^n$. Our paper is the first one which proves the spectral multiplier of a general operator on the Besov and Triebel--Lizorkin type spaces.  We will now discuss some consequences of Theorem \ref{mainthm-spectralmultipliers}.
\begin{enumerate}[{\rm (i)}]
	\item Theorem \ref{mainthm-spectralmultipliers} only requires the Gaussian upper bound  for the heat kernel of the operator $L$. This is a mild condition and allows us to apply  the results to a large number of settings ranging from the Lie groups of polynomial growth to general doubling spaces. For further details about the number of examples satisfying this condition we refer to \cite[Section 7]{DOS} and the references therein. 
	\item The extra condition $s>\f{1}{\tilde q}$ guarantees that $\|F\|_{L^\vc}\lesi F(0)+\sup_{t>0}\|\eta\, \delta_tF\|_{W^{\tilde q}_s}$ due to the embedding of the Sobolev space into bounded continuous functions. It is obvious that if $n\ge 1$, then the condition $s>\max\left\{n(\f{1}{1\wedge p\wedge  q}-\f{1}{2}),\f{1}{\tilde q}\right\}$ is simply $s> n(\f{1}{1\wedge p\wedge  q}-\f{1}{2})$. 
	\item In the particular case when $L=-\Delta$ is the standard Laplacian on $\mathbb{R}^n$, Theorem \ref{mainthm-spectralmultipliers} recovers the classical results in \cite{Tri} and the condition on the smoothness order $s$ is sharp.
	\item In the case when $\alpha=0, q=2$ and $1<p<\vc$, Theorem \ref{mainthm-spectralmultipliers} implies the boundedness of the spectral multiplier $F(\sL)$ on the space $\F^{0, L}_{p,2}(X)$ as $s>\f{n}{2}$. Notice that $\F^{0, L}_{p,2}(X)\equiv L^p(X)$ (see Remark \ref{rem1}). Hence, this recovers the result of the $L^p$-boundedness for the spectral multipliers in \cite[Theorem 3.1]{DOS}.
	\item In the case when $\alpha=0, q=2$ and $0<p\le 1$, we note that $\F^{0, L}_{p,2}(X)\equiv H^p_L(X)$ where $H^p_L(X)$ is the Hardy space associated to $L$ as in \cite{HLMMY, JY} (see Remark \ref{rem1}). In this situation, Theorem \ref{mainthm-spectralmultipliers} tells us that the spectral multiplier $F(\sL)$ is bounded on $H^p_L(X)$ as $s>n(\f{1}{p}-\f{1}{2})$. This is in line with those in \cite{B, DY}.
	\item If we propose additional conditions {\rm (H)} and {\rm (C)} on the operator $L$ as in Remark \ref{rem1}, from Theorem \ref{mainthm-spectralmultipliers} and the identification (d) in Remark \ref{rem1} we obtain the boundedness of $F(\sL)$ on the ``classical'' Besov and Triebel--Lizorkin spaces which are independent of the operator $L$. 
	To the best of our knowledge, such a result is new.
\end{enumerate}
 
 Some comments on the techniques are in order. The proof of Theorem \ref{mainthm-spectralmultipliers} relies on the following  elements:
 \begin{enumerate}[{\rm (i)}]
 	\item We obtain a  new atomic decomposition for the Triebel--Lizorkin spaces $\F^{\alpha, L}_{p,q}(X)$ for $\alpha\in \mathbb{R}$ and $0<p\le 1\le q<\vc$, see Theorem \ref{thm-new atomic decomposition}. Note that in the classical case a new atomic decomposition was proved in \cite{HPW}. 
	This kind of decomposition is similar to the atomic decomposition of the classical Hardy spaces and hence it is very useful to prove the boundedness of
	singular integrals. In particular, we are able to prove the boundedness of the spectral multiplier $F(\sL)$ on the space $\F^{\alpha, L}_{p,q}(X)$ for $\alpha\in \mathbb{R}$ and $0<p\le 1\le q<\vc$. Note that in the  case when $\alpha=0, q=2$ and $0<p\le 1$, our new atomic decomposition in  Theorem \ref{thm-new atomic decomposition} turns out to be the known results on atomic decompositions for the Hardy spaces associated to operators in \cite{HLMMY, JY}.
 	\item The duality and complex interpolation for the Triebel--Lizorkin spaces $\F^{\alpha, L}_{p,q}(X)$, see Proposition \ref{prop-duality} and Proposition \ref{prop-comple interpolation}.
 	\item The real interpolation for our new Besov and Triebel--Lizorkin spaces, see Theorem \ref{mainthm-Interpolation}. This result helps us to transfer the boundedness of the spectral multipliers $F(\sL)$ on  the Triebel--Lizorkin spaces to the boundedness on the Besov spaces. The result of Theorem \ref{mainthm-Interpolation} is also interesting in its own right. 
 \end{enumerate}
 
 The organization of this paper is as follows. In Section 2, we recall some preliminaries such as the Fefferman--Stein maximal inequality, a covering lemma, some heat kernel estimates and the definition of a recent new class of distributions. Section 3 gives definitions of the Besov and Triebel--Lizorkin spaces associated to $L$. A new atomic decomposition theorem and some results on the duality and interpolation will be addressed in this section. Finally, the proof of the main result, Theorem \ref{mainthm-spectralmultipliers},
  will be given in Section 4.
 
Throughout the paper, we always use $C$ and $c$ to denote positive constants that are independent of the main parameters involved but whose values may differ from line to line. We will write
$A\lesi B$ if there is a universal constant $C$ so that $A\leq CB$ and $A\sim B$ if $A\lesi B$ and $B\lesi A$. We write $a\vee b =\max\{a,b\}$ and $a\wedge b = \min\{a,b\}$. 

\section{Preliminaries}
\subsection{Fefferman-Stein maximal inequality and a covering lemma}
Let $0<r<\vc$. The Hardy--Littlewood maximal function $\mathcal{M}_{r}$ is defined by
	$$
	\mathcal{M}_{r} f(x)=\sup_{x\in B}\Big(\f{1}{V(B)}\int_B|f(y)|^r\dy\Big)^{1/r}
	$$
	where the sup is taken over all balls $B$ containing $x$. We will drop the subscripts $r$ when $r=1$.
	
	Let $0<r<\vc$. It is well-known that 
	\begin{equation}
	\label{boundedness maximal function}
	\|\mathcal{M}_{r} f\|_{p}\lesi \|f\|_{p}
	\end{equation}
	for all $p>r$.

	The following elementary estimates will be used frequently. See for example \cite{BDK}.
	\begin{lem}\label{lem-elementary}
		Let $\epsilon >0$.
		\begin{enumerate}[{\rm (a)}]
			\item For any $p\in [1,\vc]$ we have
			$$
			\Big(\int_X\Big[\Big(1+\f{d(x,y)}{s}\Big)^{-n-\epsilon}\Big]^p\dy\Big)^{1/p}\lesi V(x,s)^{1/p},
			$$
			for all $x\in X$ and $s>0$.
			
			\item For any $f\in L^1_{\rm loc}(X)$ we have
			$$
			\int_X\f{1}{V(x\wedge y,s)}\Big(1+\f{d(x,y)}{s}\Big)^{-n-\epsilon}|f(y)|\dy\lesi \mathcal{M}f(x),
			$$
			for all $x\in X$ and $s>0$ where $V(x\wedge y,s) =\min\{V(x,s), V(y,s)\}$.		
		\end{enumerate}
	\end{lem}
	
	We recall the Feffereman-Stein vector-valued maximal inequality and its variant in \cite{GLY}. For $0<p<\vc$, $0<q\leq \vc$ and $0<r<\min \{p,q\}$, we then have for any sequence of measurable functions $\{f_\nu\}$,
	\begin{equation}\label{FSIn}
	\Big\|\Big(\sum_{\nu}|\mathcal{M}_rf_\nu|^q\Big)^{1/q}\Big\|_{p}\lesi \Big\|\Big(\sum_{\nu}|f_\nu|^q\Big)^{1/q}\Big\|_{p}.
	\end{equation}
	The Young's inequality and \eqref{FSIn} imply the following  inequality: If $\{a_\nu\} \in \ell^{q}\cap \ell^{1}$, then 
	\begin{equation}\label{YFSIn}
	\Big\|\sum_{j}\Big(\sum_\nu|a_{j-\nu}\mathcal{M}_r f_\nu|^q\Big)^{1/q}\Big\|_{p}\lesi \Big\|\Big(\sum_{\nu}|f_\nu|^q\Big)^{1/q}\Big\|_{p}.
	\end{equation}
	\bigskip

	We will now recall  a covering lemma in \cite{C}.
	\begin{lem}\label{Christ'slemma} There
		exists a collection of open sets $\{Q_\tau^k\subset X: k\in
		\mathbb{Z}, \tau\in I_k\}$, where $I_k$ denotes certain (possibly
		finite) index set depending on $k$, and constants $\rho\in (0,1),
		a_0\in (0,1]$ and $\kappa_0\in (0,\vc)$ such that
		\begin{enumerate}[(i)]
			\item $\mu(X\backslash \cup_\tau Q_\tau^k)=0$ for all $k\in
			\mathbb{Z}$;
			\item if $i\geq k$, then either $Q_\tau^i \subset Q_\beta^k$ or $Q_\tau^i \cap
			Q_\beta^k=\emptyset$;
			\item for $(k,\tau)$ and each $i<k$, there exists a unique $\tau'$
			such $Q_\tau^k\subset Q_{\tau'}^i$;
			\item the diameter ${\rm diam}\,(Q_\tau^k)\leq \kappa_0 \rho^k$;
			\item each $Q_\tau^k$ contains certain ball $B(x_{Q_\tau^k}, a_0\rho^k)$.
		\end{enumerate}
	\end{lem}
	\begin{rem}\label{rem1-Christ}
		Since the constants $\rho$ and $a_0$ are not essential in the paper, without loss of generality, we may assume that $\rho=a_0=1/2$.  We then fix a collection of open sets in Lemma \ref{Christ'slemma} and denote this collection by $\mathscr{D}$. We call open sets in $\mathscr{D}$ the dyadic cubes in $X$ and $x_{Q_\tau^k}$ the center of the cube $Q_\tau^k \in \mathscr{D}$. We also denote 
		$$\mathscr{D}_\nu:=\{Q_\tau^{\nu+1} \in \mathscr{D} : \tau\in I_{\nu+1}\}
		$$ 
		for each $\nu\in \mathbb{Z}$. Then for $Q\in \mathscr{D}_\nu$, we have $B(x_Q, c_02^{-\nu})\subset Q\subset B(x_Q, \kappa_0 2^{-\nu}):=B_Q$ where $c_0$ is a constant independent of $Q$. From now on, let $\lambda>0$, we write $\lambda Q$ for $\lambda B_Q$ for every dyadic cube $Q$.
\end{rem}

\subsection{Heat kernel estimates}

In this section we recall some heat kernel estimates which play an important role in the proof of our                                                                                            main results.
\begin{lem}[\cite{HLMMY}]\label{lem:finite propagation}
	Let $\varphi\in \SR$ be an even function with {\rm supp}\,$\varphi\subset (-1, 1)$ and $\int \varphi =2\pi$. Denote by $\Phi$ the Fourier transform of $\varphi$. Then for every $k \in \mathbb{N}$, the kernel $K_{(t^2L)^k\Phi(t\sqrt{L})}$ of $(t^2L)^k\Phi(t\sqrt{L})$ satisfies 
	\begin{equation}\label{eq1-lemPsiL}
	\displaystyle
	{\rm supp}\,K_{(t^2L)^k\Phi(t\sqrt{L})}\subset \{(x,y)\in X\times X:
	d(x,y)\leq t\},
	\end{equation}
	and
	\begin{equation}\label{eq2-lemPsiL}
	|K_{(t^2L)^k\Phi(t\sqrt{L})}(x,y)|\leq \f{C}{V(x,t)}.
	\end{equation}
\end{lem}

Denote $V(x\vee y,s) =\max\{V(x,s), V(y,s)\}$. The following estimates are taken from \cite{CDa, BDK}.

\begin{lem}
	\label{lem1}
	\begin{enumerate}[{\rm (a)}]
		\item Let $\varphi\in \mathscr{S}(\mathbb{R})$ be an even function. Then for any $N>0$ there exists $C$ such that 
		\begin{equation}
		\label{eq1-lema1}
		|K_{\varphi(t\sqrt{L})}(x,y)|\leq \f{C}{V(x\vee y,t)}\Big(1+\f{d(x,y)}{t}\Big)^{-N},
		\end{equation}
		for all $t>0$ and $x,y\in X$.
		\item Let $\varphi_1, \varphi_2\in \mathscr{S}(\mathbb{R})$ be even functions. Then for any $N>0$ there exists $C$ such that
		\begin{equation}
		\label{eq2-lema1}
		|K_{\varphi_1(t\sqrt{L})\varphi_2(s\sqrt{L})}(x,y)|\leq C\f{1}{V(x\vee y,t)}\Big(1+\f{d(x,y)}{t}\Big)^{-N},
		\end{equation}
		for all $t\leq s<2t$ and $x,y\in X$.
		\item Let $\varphi_1, \varphi_2\in \mathscr{S}(\mathbb{R})$ be even functions with $\varphi^{(\nu)}_2(0)=0$ for $\nu=0,1,\ldots,2\ell$ for some $\ell\in\mathbb{Z}^+$. Then for any $N>0$ there exists $C$ such that
		\begin{equation}
		\label{eq3-lema1}
		|K_{\varphi_1(t\sqrt{L})\varphi_2(s\sqrt{L})}(x,y)|\leq C\Big(\f{s}{t}\Big)^{2\ell} \f{1}{V(x\vee y,t)}\Big(1+\f{d(x,y)}{t}\Big)^{-N},
		\end{equation}
		for all $t\geq s>0$ and $x,y\in X$.
	\end{enumerate}
\end{lem}

\bigskip
\begin{rem}
	(i) From \eqref{doubling3}, the term $V(x\vee y, t)$ on the right hand side of estimates in Lemma \ref{lem1} can be replaced by $V(x\vee y, d(x,y))$. 
	
	(ii) We will sometimes use the following inequality
	\[
	\Big(1+\f{d(x,y)}{t}\Big)^{-N}\Big(1+\f{d(x,z)}{t}\Big)^{-N}\leq \Big(1+\f{d(x,z)}{t}\Big)^{-N}
	\]
	for all $x,y,z\in X$ and all $t, N>0$. This inequality can be verified directly.
	
\end{rem}

\subsection{Distributions}

\bigskip

The concept of distributions has an essential role in defining the classical Besov and Triebel--Lizorkin spaces. Recently, in \cite{KP, G.etal} the authors introduced new distributions associated to a general differential operator $L$. We now recall the definition of the new distributions and some their basic properties. 

Fix $x_0\in X$ as a reference point in $X$. The class of test functions $\mathcal{S}$ associated to $L$ is defined as the set of all functions $\phi \in \cap_{m\geq 1}D(L^m)$ such that
\begin{equation}
\label{Pml norm}
\mathcal{P}_{m,\ell}(\phi)=\sup_{x\in X}(1+d(x,x_0))^m|L^\ell \phi(x)|<\vc, \ \ \forall m>0, \ell \in \mathbb{N}.
\end{equation}
It was proved in \cite{KP} that $\mathcal{S}$ is a complete locally convex space with topology generated by the family of semi-norms $\{\mathcal{P}_{m,\ell}: \, m>0, \ell \in \mathbb{N}\}$. As usual, we define  the space of distribution $\mathcal{S}'$ as the set of all continuous linear functional on $\mathcal{S}$ with the inner product defined by
\[
\langle f,\phi\rangle=f(\overline{\phi})
\]
for all $f\in \mathcal{S}'$ and $\phi\in \mathcal{S}$.

The space of distribution $\mathcal{S}'$ is suitable to define the inhomogeneous Besov and Triebel-Lizorkin spaces. Howerver, in order to study the homogeneous version of these spaces we need some modifications. 

Following \cite{G.etal} we define the set $\mathcal{S}_\vc$ as the set of all functions $\phi \in \mathcal{S}$ such that for each $k\in \mathbb{N}$ there exists $g_k\in \mathcal{S}$ so that $\phi=L^kg_k$. Note that such an $g_k$, if exists, is unique. See \cite{G.etal}.

The topology in $\mathcal{S}_\vc$ is generated by the following family of semi-norms 
\[
\mathcal{P}^*_{m,\ell,k}(\phi)=\mathcal{P}_{m,\ell}(g_k), \ \forall m>0; \ell, k\in \mathbb{N}
\]
where $\phi=L^k g_k$.

We then denote by $\mathcal{S}_\vc'$ the set of all linear functional on $\mathcal{S}_\vc$.

To see the relationship between  the spaces of distribution $\mathcal{S}'$ and $\mathcal{S}_\vc'$ we define
\[
\mathscr P_m =\{g\in\mathcal S': L^mg=0\}, m\in \mathbb{N}
\]
and set $\mathscr P =\cup_{m\in \mathbb N}\mathscr P_m$.

From Proposition 3.7 in \cite{G.etal}, we have:
\begin{prop}
	The
	following
	identification
	is
	valid $\mathcal S'/\mathscr P = \mathcal{S}_\vc'$.
\end{prop}
It was proved in \cite{G.etal} that with $L=-\Delta$, the Laplacian on $\mathbb{R}^n$,  the distributions in $\mathcal S'/\mathscr P = \mathcal{S}_\vc'$ are identical with the classical tempered distributions modulo polynomial.

\medskip

From Lemma \ref{lem1}, we can see that if $\varphi\in \mathscr{S}(\mathbb{R})$ with supp\,$\varphi\subset (0,\vc)$, then we have $K_{\varphi(t\sL)}(x,\cdot)\in \mathcal{S}_\vc$ and $K_{\varphi(t\sL)}(\cdot,y)\in \mathcal{S}_\vc$. Therefore, we can define
\begin{equation}\label{eq- s and s'}
\varphi(t\sL)f(x)=\langle f, K_{\varphi(t\sL)}(x,\cdot)\rangle
\end{equation}
for all $f\in \mathcal{S}'_\vc$.

The support condition supp\,$\varphi\subset (0,\vc)$ is essential so that one can define $\varphi(t\sL)f$ with $f\in \mathcal{S}'_\vc$. In general, if $\varphi\in \mathscr{S}(\mathbb{R})$, then we have $K_{\varphi(t\sL)}(x,\cdot)\in \mathcal{S}$ and $K_{\varphi(t\sL)}(\cdot,y)\in \mathcal{S}$. In this situation, it is able to define $\varphi(t\sL)f$ with $f\in \mathcal{S}'$, but it is not clear how to define $\varphi(t\sL)f$ with $f\in \mathcal{S}'_\vc$.

	
\section{Besov and Triebel--Lizorkin spaces associated to operators}

This section devotes to the definition of the Besov and Triebel--Lizorkin spaces associated to operators and their properties such as square function characterization, atomic decomposition, duality and interpolation.

\subsection{Definitions of Besov spaces $\B^{\alpha,L}_{p,q}$ and Triebel--Lizorkin spaces $\F^{\alpha,L}_{p,q}$.}

In what follows, by a ``partition of unity'' we shall mean that a function $\psi\in \mathcal{S}(\mathbb{R})$ such that $\supp\psi\subset[1/2,2]$, $\int\psi(\xi)\,\f{d\xi}{\xi}\neq 0$ and
$$\sum_{j\in \mathbb{Z}}\psi_j(\lambda)=1 \textup{ on } (0,\infty)$$
where $\psi_j(\lambda):=\psi(2^{-j}\lambda)$ for each $j\in \mathbb{Z}$.

We now recall the definition of Besov and Triebel--Lizorkin spaces associated to the operator $L$ in \cite[Definition 3.1]{BBD} (see also \cite{G.etal}). 
\begin{defn}
	Let $\psi$ be  a partition of unity. For $0< p, q\leq \vc$ and $\alpha\in \mathbb{R}$, we define the homogeneous Besov space $\B^{\alpha, \psi, L}_{p,q,w}(X)$ as follows 
	$$
	\Big\{f\in \mathcal{S}'_\vc:  \|f\|_{\B^{\alpha, \psi, L}_{p,q}(X)}<\vc\}
	$$
	where
	\[
	\|f\|_{\B^{\alpha, \psi, L}_{p,q}(X)}= \Big\{\sum_{j\in \mathbb{Z}}\left(2^{j\alpha}\|\psi_j(\sqrt{L})f\|_{p}\right)^q\Big\}^{1/q}.
	\]
	
	Similarly, for $0< p<\vc$, $0<q\le \vc$ and $\alpha\in \mathbb{R}$, the homogeneous Triebel-Lizorkin space $\F^{\alpha, \psi, L}_{p,q}(X)$ is defined by 
	$$
	\Big\{f\in \mathcal{S}'_\vc:  \|f\|_{\F^{\alpha, \psi, L}_{p,q}(X)}<\vc\}
	$$
	where
	\[
	\|f\|_{\F^{\alpha, \psi, L}_{p,q}(X)}= \Big\|\Big[\sum_{j\in \mathbb{Z}}(2^{j\alpha}|\psi_j(\sqrt{L})f|)^q\Big]^{1/q}\Big\|_{p}.
	\]
\end{defn}

It was proved in \cite[Proposition 3.2]{BBD} (see also \cite{G.etal}) that Besov spaces and Triebel--Lizorkin spaces defined as above are independent of the choices of partition of identity functions. More precisely, we have:
\begin{thm}
	\label{thm1}
	Let $\psi, \varphi\in C^\vc_c(\mathbb{R})$ be partitions of unity. Then we have:
	\begin{enumerate}[(a)]
		\item  The spaces $\B^{\alpha, \psi, L}_{p,q}(X)$ and $\B^{\alpha, \varphi, L}_{p,q}(X)$ coincide with equivalent norms for all $0< p, q\le \vc$ and $\alpha\in \mathbb{R}$.
		
		\item  The spaces $\F^{\alpha, \psi, L}_{p,q}(X)$ and $\F^{\alpha, \varphi, L}_{p,q}(X)$ coincide with equivalent norms for all $0< p< \vc$, $0<q\le \vc$ and $\alpha\in \mathbb{R}$.		
	\end{enumerate}
	For this reason, we define the spaces  $\B^{\alpha, L}_{p,q}(X)$ and  $\F^{\alpha, L}_{p,q}(X)$ to be any spaces  $\B^{\alpha, \psi, L}_{p,q}(X)$ and  $\B^{\alpha, \psi, L}_{p,q,w}(X)$ with any partitions of unity $\psi$, respectively. 
\end{thm}

In is interesting that like the classical case, our new spaces can be characterized in terms of the square functions.

For $\alpha\in \mathbb{R}, \lambda, a>0$ and $0<q<\vc$ we define the Lusin function and the Littlewood--Paley function by setting
\begin{equation}\label{g-function}
\mathcal{G}^{\alpha}_{\lambda, q}F(x)=\left[\int_0^\vc\int_X(t^{-\alpha}|F(y,t)|)^q\Big(1+\f{d(x,y)}{t}\Big)^{-\lambda q} \f{\dy dt}{tV(x,t)}\right]^{1/q}
\end{equation}
and
\begin{equation}\label{lusin-function}
\mathcal{S}^\alpha_{a,q}F(x)=\left[\int_0^\vc\int_{d(x,y)<at}(t^{-\alpha}|F(y,t)|)^q\f{\dy dt}{tV(x,t)}\right]^{1/q},
\end{equation}
respectively.

When either $\alpha=0$ or $a=1$ we will drop them in the notation of $\mathcal{S}^\alpha_{a,q}$ and $\mathcal{G}^\alpha_{\lambda,q}$. We now have the following result regarding the estimates on the change of the angles for the function $\mathcal{S}^\alpha_{a,q}$. See \cite[Proposition 3.11]{BBD}.

\begin{prop}
	\label{pro-change the angle}
	Let $a>1$,  $0<p,q<\vc$ and $\alpha\in \mathbb{R}$. Then there exists a constant $C>0$ so that
	\[
	\|\mathcal{S}^\alpha_{a,q}F\|_{p}\leq Ca^{\f{n}{p\wedge q}}\|\mathcal{S}^\alpha_{q}F\|_{p}
	\]
	for all $F$.
\end{prop}

We now recall square function characterizations for our new Triebel--Lizorkin spaces in Proposition \ref{prop-square function characterization} and Proposition \ref{prop4.1-heat kernel} below which are taken from Proposition 3.13 and Proposition 3.14 in \cite{BBD}, respectively.
\begin{prop}
	\label{prop-square function characterization}
	Let $\psi$ be a partition of unity. Then for   $0<p, q<\vc$, and $\alpha\in \mathbb{R}$, we have
	$$
	\|f\|_{\F^{\alpha,L}_{p,q}}\sim  \|\mathcal{G}^{\alpha}_{\lambda, q}(\psi(t\sL)f)\|_{p}\sim \|\mathcal{S}^{\alpha}_{q}(\psi(t\sL)f)\|_{p}
	$$
	for all $f\in \mathcal{S}'_\vc$, provided that $\lambda>\f{n}{p\wedge q}$.
\end{prop}

For each $m\in \mathbb{N}$ we denote by $\mathscr{S}_m(\mathbb{R})$ the set of all even functions $\varphi\in \mathscr{S}(\mathbb{R})$ such that $\varphi(\xi)=\xi^{2m}\phi(\xi)$ for some $\phi\in \mathscr{S}(\mathbb{R})$ and $\varphi(\xi)\neq 0$ on $(-2,-1/2)\cup(1/2,2)$.

We also have a similar square function characterization for new Triebel-Lizorkin spaces via functions in $\mathscr{S}_m(\mathbb{R})$.
\begin{prop}\label{prop4.1-heat kernel}
	Let $0<p<\vc$, $0<q< \vc$, $\alpha\in \mathbb{R}$, $\lambda>0$ and $\varphi\in \mathscr{S}_m(\mathbb{R})$ with $m>\alpha/2$, we have
	$$
	\|f\|_{\F^{\alpha,L}_{p,q}}\sim  \|\mathcal{G}^{\alpha}_{\lambda, q}(\varphi(t\sL)f)\|_{p}\sim \|\mathcal{S}^{\alpha}_{q}(\varphi(t\sL))f\|_{p}
	$$
	for all $f\in L^2(X)$ provided that $\lambda>\f{n}{p\wedge q}$.
\end{prop}
\begin{rem}\label{rem0}
	The condition $\varphi(\xi)\neq 0$ on $(-2,-1/2)\cup(1/2,2)$ for the class $\mathscr{S}_m(\mathbb{R})$ in Proposition \ref{prop4.1-heat kernel} can be replaced by $\varphi(\xi)\neq 0$ on $(-2a,-a/2)\cup(a/2,2a)$ for some $a>0$ due to the fact that
	\[
	 \|\mathcal{G}^{\alpha}_{\lambda, q}(\varphi(t\sL)f)\|_{p}\sim \|\mathcal{S}^{\alpha}_{q}(\varphi(t\sL))f\|_{p}\sim  \|\mathcal{G}^{\alpha}_{\lambda, q}(\varphi(at\sL)f)\|_{p}\sim \|\mathcal{S}^{\alpha}_{q}(\varphi(at\sL))f\|_{p,w}
	\]
\end{rem}
\begin{rem}\label{rem1}
We recall some identifications of the Besov and Triebel--Lizorkin spaces with certain known function spaces in \cite[Section 5]{BBD}:
\begin{enumerate}[{\rm (a)}]
	\item $\F^{0,L}_{p,2}(X)\equiv L^p(X)$ for $1<p<\vc$;
	\item $\F^{0,L}_{p,2}(X)\equiv H^p_L(X)$ for $0<p\le 1$ where $H^p_L(X)$ is the Hardy space associated to $L$ as in \cite{HLMMY, JY};
	\item $\F^{0,L}_{\vc,2}(X)\equiv BMO_L(X)$ where $BMO_L(X)$ is the BMO space associated to $L$ as in \cite{DY1};
	\item If $L$ satisfies two following addition conditions: 
	\begin{enumerate}
		\item[\rm (H)] There exists $\delta \in (0,1]$  so that \\
		\begin{equation}\label{eq:H}
		| p_t(x,y) - p_t(\bar x,y) | \lesssim \Big(\f{d(x,\bar x)}{\sqrt{t}}\Big)^{\delta _0}\frac{1}{{V\left( {x,\sqrt t } \right)}}\exp\Big(-\f{d(x,y)^2}{ct}\Big)
		\end{equation}
		whenever $d(x,\bar x)< \sqrt{t}; $
		\item[\rm (C)] $\displaystyle \int_X {{p_t}\left( {x,y} \right)d\mu \left( x \right)}  = 1$ for all $y \in X$ and $t>0$,
	\end{enumerate}
then $\F^{\alpha,L}_{p,q}(X)=\F^{\alpha}_{p,q}(X)$ and $\B^{\alpha,L}_{p,q}(X)=\B^{\alpha}_{p,q}(X)$ for $\frac{n}{n+\delta_0}<p,q<\infty$ and $\frac{n}{p\wedge q}-n-\delta_0<\alpha<\delta_0$ where $\B^{\alpha}_{p,q}(X)$ and $\F^{\alpha}_{p,q}(X)$ are the Besov and Triebel--Lizorkin spaces defined as in \cite{HMY, HS}.
\end{enumerate} 
\end{rem}

\subsection{Atomic decompositions for Besov spaces $\B^{\alpha,L}_{p,q}$ and Triebel--Lizorkin spaces $\F^{\alpha,L}_{p,q}$.}
In this section, we recall atomic decomposition theorems for our new Besov and Triebel-Lizorkin spaces in \cite{BBD}.

\begin{defn}\label{defLmol}
	Let $0< p\leq \vc$ and $M\in \mathbb{N}_+$. A function $a$ is said to be an $(L, M, p)$ atom if there exists a dyadic cube $Q\in \mathscr{D}$ so that
	\begin{enumerate}[{\rm (i)}]
		\item $a=L^{M} b$;
		
		\item ${\rm supp} \,L^{k} b\subset 3B_Q$, $k=0,\ldots , 2M$;
		
		\item $\displaystyle |L^{k} b(x)|\leq \ell(Q)^{2(M-k)}V(Q)^{-1/p}$, $k=0,\ldots , 2M$;
	\end{enumerate}
	where $B_Q$ is a ball associated to $Q$ defined in Remark \ref{rem1-Christ}.
\end{defn}

The following results on the atomic decompositions for the Besov and Triebel--Lizorkin are taken from Theorem 4.2, Theorem 4.3, Theorem 4.6 and Theorem 4.7 in \cite{BBD}, respectively.
\begin{thm}\label{thm1- atom Besov}
	Let $\alpha\in \mathbb{R}$ and  $0<p,q\leq \vc$. For  $M\in \mathbb{N}_+$, if $f\in \B^{\alpha,L}_{p,q}(X)$ then there exist a sequence of $(L,M,p)$ atoms $\{a_Q\}_{Q\in \mathscr{D}_\nu, \nu\in \mathbb{Z}}$ and a sequence of coefficients  $\{s_Q\}_{Q\in \mathscr{D}_\nu, \nu\in\mathbb{Z}}$ so that
	$$
	f=\sum_{\nu\in\mathbb{Z}}\sum_{Q\in \mathscr{D}_\nu}s_Qa_Q \ \ \text{in $\mathcal{S}'_\vc$}.
	$$
	Moreover,
	$$
	\Big[\sum_{\nu\in\mathbb{Z}}2^{\nu \alpha q}\Big(\sum_{Q\in \mathscr{D}_\nu}|s_Q|^p\Big)^{q/p}\Big]^{1/q}\sim \|f\|_{\B^{\alpha,L}_{p,q}(X)}.
	$$
\end{thm}
Conversely, each atomic decomposition with suitable coefficients belong to the spaces  $\B^{\alpha,L}_{p,q}(X)$.

\begin{thm}\label{thm2- atom Besov}
	Let $\alpha\in \mathbb{R}$ and $0<p,q\leq \vc$. If
	$$
	f=\sum_{\nu\in\mathbb{Z}}\sum_{Q\in \mathscr{D}_\nu}s_Qa_Q \ \ \text{in $\mathcal{S}'_\vc$}
	$$
	where $\{a_Q\}_{Q\in \mathscr{D}_\nu, \nu\in \mathbb{Z}}$ is a sequence of $(L,M,p)$ atoms and $\{s_Q\}_{Q\in \mathscr{D}_\nu, \nu\in\mathbb{Z}}$ is a sequence of coefficients satisfying
	$$
	\Big[\sum_{\nu\in\mathbb{Z}}2^{\nu\alpha q}\Big(\sum_{Q\in \mathscr{D}_\nu}|s_Q|^p\Big)^{q/p}\Big]^{1/q}<\vc,
	$$
	then $f\in \B^{\alpha,L}_{p,q}(X)$ and
	$$
	\|f\|_{\B^{\alpha,L}_{p,q}(X)} \lesi \Big[\sum_{\nu\in\mathbb{Z}}\Big(\sum_{Q\in \mathscr{D}_\nu}|s_Q|^p\Big)^{q/p}\Big]^{1/q}
	$$
	as long as $M>\f{n}{2}+\f{1}{2}\max\{\alpha,\f{n}{1\wedge p\wedge q}-\alpha\}$.
\end{thm}	

The similar results also hold for the Triebel--Lizorkin spaces $\F^{\alpha,L}_{p,q}(X)$.
\begin{thm}\label{thm1- atom TL spaces}
	Let $\alpha\in \mathbb{R}$, $0<p<\vc$ and $0<q\leq \vc$. If $f\in \F^{\alpha,L}_{p,q}(X)$ then there exist a sequence of $(L,M,p)$ atoms $\{a_Q\}_{Q\in \mathscr{D}_\nu, \nu\in \mathbb{Z}}$ and a sequence of coefficients  $\{s_Q\}_{Q\in \mathscr{D}_\nu, \nu\in\mathbb{Z}}$ so that
	$$
	f=\sum_{\nu\in\mathbb{Z}}\sum_{Q\in \mathscr{D}_\nu}s_Qa_Q \ \ \text{in $\mathcal{S}'_\vc$}.
	$$
	Moreover,
	\begin{equation}\label{eq1-thm1 atom TL space}
	\Big\|\Big[\sum_{\nu\in\mathbb{Z}}2^{\nu\alpha q}\Big(\sum_{Q\in \mathscr{D}_\nu}V(Q)^{-1/p}|s_Q|\chi_Q\Big)^q\Big]^{1/q}\Big\|_{p}\lesi \|f\|_{\F^{\alpha,L}_{p,q}}.
	\end{equation}
\end{thm}

\begin{thm}\label{thm2- atom TL spaces}	
	Let $\alpha\in \mathbb{R}$, $0<p<\vc$ and $0<q\le \vc$. If
	$$
	f=\sum_{\nu\in\mathbb{Z}}\sum_{Q\in \mathscr{D}_\nu}s_Qa_Q \ \ \text{in $\mathcal{S}'_\vc$}
	$$
	where $\{a_Q\}_{Q\in \mathscr{D}_\nu, \nu\in \mathbb{Z}}$ is a sequence of $(L,M,p)$ atoms and $\{s_Q\}_{Q\in \mathscr{D}_\nu, \nu\in\mathbb{Z}}$ is a sequence of coefficients satisfying
	$$
	\Big\|\Big[\sum_{\nu\in\mathbb{Z}}2^{\nu\alpha q}\Big(\sum_{Q\in \mathscr{D}_\nu}V(Q)^{-1/p}|s_Q|\chi_Q\Big)^q\Big]^{1/q}\Big\|_{p}<\vc,
	$$
	then $f\in \F^{\alpha,L}_{p,q}$ and
	$$
	\|f\|_{\F^{\alpha,L}_{p,q}} \lesi \Big\|\Big[\sum_{\nu\in\mathbb{Z}}2^{\nu\alpha q}\Big(\sum_{Q\in \mathscr{D}_\nu}V(Q)^{-1/p}|s_Q|\chi_Q\Big)^q\Big]^{1/q}\Big\|_{p,w}
	$$
	provided $M>\f{n}{2}+\f{1}{2}\max\{\alpha,\f{n}{1\wedge p\wedge q}-\alpha\}$.
\end{thm}

\begin{rem}
	From the atomic decomposition results above, it is easy to see that $\B^{\alpha,L}_{p,q}(X)\cap L^2(X)$ and $\F^{\alpha,L}_{p,q}(X)\cap L^2(X)$ are dense in $\B^{\alpha,L}_{p,q}(X)$ and $\F^{\alpha,L}_{p,q}(X)$ for all $\alpha\in \mathbb{R}$ and $0<p,q<\vc$, respectively.
\end{rem}
\subsection{New atomic decompositions for Triebel--Lizorkin spaces $\F^{\alpha,L}_{p,q}$.}

In order to prove the sharp estimate for the spectral multipliers on Triebel--Lizorkin spaces $\F^{\alpha,L}_{p,q}$, the atomic decomposition results  in Theorems \ref{thm1- atom Besov}--\ref{thm2- atom TL spaces} are not sufficient. To overcome this trouble we prove a new atomic decomposition theorem for the new Triebel-Lizorkin spaces. This kind of atomic decomposition is quite similar to the atomic decomposition of the Hardy spaces. We note that such an atomic decomposition for the classical Triebel--Lizorkin spaces was proved in \cite{HPW}. We first introduce the new definition of atoms related to $L$.

\begin{defn}\label{defLmol}
	Let $\alpha\in \mathbb{R}$, $0< p\leq 1\le q <\vc$ and $M\in \mathbb{N}_+$. A function $a$ is said to be a new $(L, M, \alpha, p, q)$ atom if there exists a ball $B$ so that
	\begin{enumerate}[{\rm (i)}]
		\item $a=L^{M} b$;
		
		\item ${\rm supp} \,L^{k} b\subset B$, $k=0,\ldots , M$;
		
		\item $\displaystyle \|L^{k} b\|_{\F^{\alpha, L}_{q,q}(X)}\leq r_Q^{2(M-k)}|B|^{\f{1}{q}-\f{1}{p}}$, $k=0,\ldots, M$.
	\end{enumerate}
\end{defn}
In the particular case, the $(L, M, 0, p,2)$--atom coincide with the notions of atoms in the Hardy spaces $H^p_L$ associated to operators $L$  for $0<p\le 1$, which were considered in \cite{HLMMY, JY}.

\begin{thm}
	\label{thm-new atomic decomposition}
	Let $\alpha\in \mathbb{R}$ and  $0<p\le 1< q<\vc$. If $f\in \F^{\alpha,L}_{p,q}(X)\cap L^2(X)$ then there exist a sequence of new $(L, M, \alpha, p, q)$ atoms $\{a_j\}_j$ and a sequence of coefficients  $\{\lambda_j\}_j$ so that
$$
f=\sum_j \lambda_ja_j \ \ \text{in $L^2(X)$}.
$$
Moreover,
\begin{equation}\label{eq1-thm new atomic decomposition}
\Big(\sum_j |\lambda_j|^p\Big)^{1/p}\lesi \|f\|_{\F^{\alpha,L}_{p,q}}.
\end{equation}

Conversely, if $f\in L^2(X)$ such that 
$$
f=\sum_j \lambda_ja_j \ \ \text{in $L^2(X)$}
$$
where $\{a_j\}_j$ is a sequence of new $(L, M, \alpha, p, q)$ atoms with $M>\f{n}{2p}$ and $\{\lambda_j\}_j\in \ell^p$, then $f\in \F^{\alpha,L}_{p,q}(X)$ and
\begin{equation}\label{eq2-thm new atomic decomposition}
\|f\|_{\F^{\alpha,L}_{p,q}}\lesi \Big(\sum_j |\lambda_j|^p\Big)^{1/p}.
\end{equation}
\end{thm}

\begin{proof}
	
(a)	Let $\psi$ be a partition of unity and $\Phi$ be as in Lemma \ref{lem:finite propagation}. Setting, $\Psi(t)=t^{2m}\Phi(t)$ with $m\in \mathbb{N}$ and $m>|\alpha|/4$. Then by the spectral theory, for each $f \in \F^{\alpha, L}_{p,q}(X)\cap L^2(X)$,
	\begin{equation}
	f=c_{\Psi,\psi}\int_0^\vc(t\sL)^{2M}\Psi(t\sqrt{L})\psi(t\sL)f\f{dt}{t}
	\end{equation}
	with the integral converges in $L^2(X)$ where 
	\[
	c_{\Psi,\psi}=\Big[\int_{0}^{\vc}(s)^{2M}\Psi(s)\psi(s)\f{ds}{s}\Big]^{-1}.
	\]	
	
	For each dyadic cube $Q$ in $X$, denote by $\ell(Q)$ the sidelength of $Q$ and
	$$
	Q^+=\{(x,t): x\in Q, \ell(Q)/2<t\leq \ell(Q)\}.
	$$
	As in \cite{CF}, for each $k\in \mathbb{Z}$, we set
	$$
	O_k=\{x: (\mathcal{S}^\alpha_{c_0,q,\psi}f(x))^p>2^{k}\}=\{x: \mathcal{S}^\alpha_{c_0,q,\psi}f(x)>2^{k/p}\},
	$$
	and
	$$\mathscr{A}_k=\{Q\in \mathscr{D}: \mu(Q\cap O_k)>\mu(Q)/2\geq  \mu(Q\cap O_{k+1})\},
	$$
	where $c_0$ is a positive constant which will be fixed later, and 
	$$\mathcal{S}^\alpha_{c_0,q,\psi}f(x)=\Big(\int\int_{d(x,y)<c_0 t}(t^{-\alpha}|\psi(tL)f(y)|)^q\f{\dy dt}{tV(x,t)}\Big)^{1/q}
	$$
	and $\mathscr{D}$ is the collection of all dyadic cubes.
	
	For each $k\in \mathbb{Z}$, denote by $\{Q_k^l\}$ the maximal dyadic cubes in $\mathscr{A}_k$. It is easy to see that for each dyadic cube in $X$ there is a unique $k\in \mathbb{Z}$ so that $Q\in \mathscr{A}_k$. Therefore, we can write
	\begin{equation*}
	\begin{aligned}
	f&=\sum_{k,l}c_{\Psi,\psi} \sum\limits_{Q\su Q_k^l; Q\in \mathscr{A}_k}\int\int_{Q^+}t^{2M}L^M K_{\Psi(t\sqrt{L})}(x,y)\psi(t\sL)f(y)\dy\f{dt}{t}\\
	&=\sum_{k,l}\lambda_{k,l}a_{k,l}
	\end{aligned}
	\end{equation*}
	where $a_{k,l}=L^Mb_{k,l}$,
	$$
	b_{k,l}=\f{c_{\Psi,\psi}}{\lambda_{k,l}}\sum\limits_{Q\su Q_k^l; Q\in \mathscr{A}_k}\int\int_{Q^+}t^{2M}K_{\Phi(t\sqrt{L})}(x,y)\psi(t\sL)f(y)\dy\f{dt}{t}
	$$
	and
	$$
	\lambda_{k,l}= \mu(Q_k^l)^{\f{1}{p}-\f{1}{q}}\Big(\sum\limits_{Q\su Q_k^l; Q\in \mathscr{A}_k} (t^{-\alpha}|\psi(t\sL)f(y)|)^q\dy\f{dt}{t}\Big)^{1/q}.
	$$
	Note  that for $j=0,1,\ldots, M$ we can write
	$$
	L^jb_{k,l}=\f{c_{\Psi,\psi}}{\lambda_{k,l}}\sum\limits_{Q\su Q_k^l; Q\in \mathscr{A}_k}\int\int_{Q^+}t^{2(M-j)}K_{(t^2L)^j\Psi(t\sqrt{L})}(x,y)\psi(t\sL)f(y)\dy\f{dt}{t}.
	$$
	Moreover, due to $y\in Q\su Q_k^l$ and Lemma \ref{lem:finite propagation}, {\rm supp}$L^jb_{k,l}\su 3Q_k^l$ for all $j=0,1,\ldots, M$.
	Furthermore, for any $h\in \F^{-\alpha, L}_{q',q'}(X)$ supported in $3Q_k^l$ with norm $\|h\|_{\F^{-\alpha, L}_{q',q'}(X)}=1$, we have
	\begin{equation*}
	\begin{aligned}
	\Big|\int &(\ell(Q_k^l)^2L)^jb_{k,l}(x)h(x)\dx\Big|\\
	&= \f{c_{\Psi,\psi}}{\lambda_{k,l}}\Big|\sum\limits_{Q\su Q_k^l; Q\in \mathscr{A}_k}\int\int_{Q^+}t^{2M}K_{(t^2L)^j\Psi(t\sqrt{L})}(x,y)\psi(t\sL)f(y)\dy\f{dt}{t} h(x)\dx  \Big|\\
	&\leq C\f{\ell(Q_k^l)^{2M}}{\lambda_{k,l}}\Big[\sum\limits_{Q\su Q_k^l; Q\in \mathscr{A}_k}\int\int_{Q^+}(t^{-\alpha}|\psi(t\sL)f(y)|)^q\dy\f{dt}{t} \Big]^{1/q}\\
	&~~~~~~~~~~\times\Big[\int_{\mathbb{R}^{n+1}_+}(t^{\alpha}|(t^2L)^j\Psi(t\sqrt{L})h(y)|)^{q'}\dy\f{dt}{t}\Big]^{1/q'} .\\
	\end{aligned}
	\end{equation*}
	Note that since $\Phi(0)=1$, $\xi^{2j}\Psi(\xi)\neq 0$ for $\epsilon/2<|\xi|<2\epsilon$ for some $\epsilon>0$. Hence, by Proposition \ref{prop-duality} and Remark \ref{rem0}, we have
	\[
	\Big[\int_{\mathbb{R}^{n+1}_+}(t^{\alpha}|(t^2L)^j\Psi(t\sqrt{L})h(y)|)^{q'}\dy\f{dt}{t}\Big]^{1/q'}\lesi \|h\|_{\F^{-\alpha, L}_{q',q'}(X)}.
	\]
	Therefore,
	\begin{equation*}
	\begin{aligned}
	\Big|\int (\ell(Q_k^l)^2L)^jb_{k,l}(x)h(x)\dx\Big|
	&\leq C\f{\ell(Q_k^l)^{2M}}{\lambda_{k,l}}\Big[\sum\limits_{Q\su Q_k^l; Q\in \mathscr{A}_k}\int\int_{Q^+}(t^{-\alpha}|\psi(t\sL)f(y)|)^qdy\f{dt}{t} \Big]^{1/q}\\
	&\leq C\ell(Q_k^l)^{2M}\mu(Q_{k}^l)^{1/q-1/p}.
	\end{aligned}
	\end{equation*}
	Therefore, for all $j=0,1,\ldots, M$,
	$$
	\|(\ell(Q_k^l)^2L)^jb_{k,l}\|_{\F^{\alpha,L}_{q,q}(X)}\leq C\ell(Q_k^l)^{2M}\mu(Q_{k}^l)^{1/q-1/p}
	$$
	and hence $b_{k,l}$'s are $(L, M, \alpha, p, q)$-atoms up to a normalization by a multiplicative constant.\\
	
	To complete our proof, we need to check that $\sum_{k,l}\lambda_{k,l}^p\leq C\|f\|^p_{\F^{\alpha,L}_{p,q}(X)}$. To do this, denote by $\mathcal M$ the Hardy--Littlewood maximal function. For each $k$, we define 
	\[
	O_k^*=\{x: \mathcal M(\chi_{O_k})(x)>1/2\}.
	\]
	From the weak type of $(1,1)$ of the maximal function $\mathcal M$ we have $\mu(O_k^*)\lesi \mu(O_k)$. Moreover, from the definition of $\mathscr A_k$, we have $Q\subset O_k^*$ as $Q\in \mathscr A_k$. Note that there exists $c_0>1$ so that for any $x\in Q$ and $(y,t)\in Q^+$, $d(x,y)\le c_0t$. For these reasons, for every $k$, we have
	\[
	\begin{aligned}
	\sum_{l}  \sum\limits_{Q\su Q_k^l; Q\in \mathscr{A}_k}&\int\int_{Q^+}(t^{-\alpha}|\psi(t\sL)f(y)|)^q\dy \f{dt}{t}\\
	&\le \sum_{l}  \sum\limits_{Q\su Q_k^l; Q\in \mathscr{A}_k}\int_X\int_{Q^+}\chi_{O_k^*\backslash O_{k+1}}(x)\chi\Big(\f{d(\cdot,y)}{c_0t}\Big)(x)(t^{-\alpha}|\psi(t\sL)f(y)|)^q\f{\dx}{\mu(Q)}\dy \f{dt}{t}\\
	&\sim \sum_{l}  \sum\limits_{Q\su Q_k^l; Q\in \mathscr{A}_k}\int_{Q^+}\int_X\chi_{O_k^*\backslash O_{k+1}}(x)\chi\Big(\f{d(\cdot,y)}{c_0t}\Big)(x)(t^{-\alpha}|\psi(t\sL)f(y)|)^q\f{\dx}{V(x,t)} \dy\f{dt}{t}\\
	&\lesi \int_{O_k^*\backslash O_{k+1}} \mathcal{S}^\alpha_{c_0,q,\psi}f(x)\dx \lesi 2^{kq/p}\mu (O_k^*).
	\end{aligned}
	\] 
	This, along with the fact that $\mu(O_k^*)\lesi \mu(O_k)$, implies that
	\begin{equation}\label{eq1}
	\sum_{l}  \sum\limits_{Q\su Q_k^l; Q\in \mathscr{A}_k}\int\int_{Q^+}(t^{-\alpha}|\psi(t\sL)f(y)|)^q\dy \f{dt}{t} \leq 2^{qk/p}w(O_k)
	\end{equation}
	for every $k$.
	
	With this estimate on hand, applying H\"older inequality and (\ref{eq1}), one has
	\begin{equation*}
	\begin{aligned}
	\sum_{k,l}\lambda_{k,l}^p&=\sum_{k,l}\mu(Q_k^l)^{1-p/q}\Big(\sum\limits_{Q\su Q_k^l; Q\in \mathscr{A}_k}\int\int_{Q^+} |\psi(t\sL)f(y)|^qdy\f{dt}{t}\Big)^{p/q}\\
	&=\sum_k\Big[\sum_l \Big(\mu(Q_k^l)^{1-\f{p}{q}}\Big)^{\f{q}{q-p}}\Big]^{\f{q-p}{q}}\Big({\sum_l \sum\limits_{Q\su Q_k^l; Q\in \mathscr{A}_k}\int\int_{Q^+} (t^{-\alpha}|\psi(t\sL)f(y)|)^qdy\f{dt}{t}}\Big)^{p/q}\\
	&\leq \sum_k \mu(O_k)^{\f{q-p}{q}}2^k\mu(O_k)^{p/q}\\
	&\leq \sum_k 2^k\mu(O_k)\\
	&\leq C\|(\mathcal{S}^\alpha_{c_0,q,\psi}f)^p\|_{1} \sim \|f\|^p_{\F^{\alpha,L}_{p,q}}
	\end{aligned}
	\end{equation*}
	where in the last inequality we used Propositions \ref{pro-change the angle} and Proposition \ref{prop-square function characterization}.

	(b) Conversely, it suffices that prove that there exists $C>0$ such that 
	\begin{equation}\label{eq1-proof new atomic decomposition}
	\|a\|_{\F^{\alpha, L}_{p,q}(X)}\le C
	\end{equation}
	for each $(L,M,\alpha, p,q)$ atom $a$ associated to some ball $B\subset X$.
	
	Let $\Phi$ be a function as in Lemma \ref{lem:finite propagation} and $\Psi(\xi):=|\xi|^{2m}\Phi(\xi)$ with $m>|\alpha|/4$ as in (a) so that $\Psi(\xi)\ne 0$ on $\{\xi: \epsilon/2\le |\xi|\le 2\epsilon\}$ for some $\epsilon>0$. By Proposition \ref{prop4.1-heat kernel} and Remark \ref{rem0}, we use \eqref{eq1-proof new atomic decomposition}  to prove that 
	\begin{equation}\label{eq2-proof atomic decom}
	\Big\|\Big(\int_0^\vc(t^{-\alpha}|\Psi(t\sL)a|)^q\f{dt}{t}\Big)^{1/q}\Big\|_{p}\le C.
	\end{equation}
	
	To do this, we first note that from Lemma \ref{lem:finite propagation},
	\begin{equation}\label{eq-finite propagation epsilon=1}
	K_{\Psi(t\sL)}(\cdot,\cdot)\subset  \{(x,y)\in X\times X:
	d(x,y)\leq t\}.
	\end{equation}
	
	We then write
	\[
	\begin{aligned}
	\Big\|\Big(\int_0^\vc(t^{-\alpha}|\Psi(t\sL)a|)^q\f{dt}{t}\Big)^{1/q}\Big\|_{p}&\le \Big\|\Big(\int_0^\vc(t^{-\alpha}|\Psi(t\sL)a|)^q\f{dt}{t}\Big)^{1/q}\Big\|_{L^p(8B)}\\
	& \ \ \ \ +\Big\|\Big(\int_0^\vc(t^{-\alpha}|\Psi(t\sL)a|)^q\f{dt}{t}\Big)^{1/q}\Big\|_{L^p((8B)^c)}=:I_1 + I_2.
	\end{aligned}
	\]
	Using H\"older's inequality, we obtain that
	\[
	I_1\lesi \|a\|^p_{\F^{\alpha,L}_{q,q}(X)}V(B)^{1-p/q}\lesi 1.
	\]
	For the second term $I_2$, assume that $a=L^Mb$. We note that for $x\in (8B)^c$, by using Lemma \ref{lem1}, $\Psi(t\sL)a(x)=0$ as $t<d(x,x_B)/4$ and 
	\[
	\begin{aligned}
	|\Psi(t\sL)a(x)|=|L^M\Psi(t\sL)b(x)| &\lesi t^{-2M}\int_X K_{(t^2L)^M\Psi(t\sL)}(x,y)|b(y)|\dy\\
	&=\lesi t^{-2M}\int_X K_{(t^2L)^M\Psi^*(t\sL)}(y,x)|b(y)|\dy\\
	&\lesi t^{-2M}\|b\|_{\F^{\alpha,L}_{q,q}(X)}\|K_{(t^2L)^M\Psi^*(t\sL)}(\cdot,x)\|_{\F^{\alpha,L}_{q',q'}(X)}\\
	&\lesi \Big(\f{r_B}{t}\Big)^{2M}V(B)^{1/q-1/p}\|K_{(t^2L)^M\Psi^*(t\sL)}(\cdot,x)\|_{\F^{\alpha,L}_{q',q'}(X)}
	\end{aligned}
	\]
	as $t\ge d(x,x_B)/4$ where $\Psi^*$ is the conjugate of $\Psi$.
	
	It is easy to see that $K_{(t^2L)^M\Psi^*(t\sL)}(\cdot,x)=L^M \tilde{b}$ where $\tilde{b}=t^{2M}\Psi^*(t\sL)(x,\cdot)=t^{2M}\Psi(t\sL)(x,\cdot)$.
	
	Recalling Lemma \ref{lem:finite propagation} and the fact that $t>d(x,x_B)/4$, we have, for each $j=0,1,\ldots,M$,
	\[
	L^j \tilde{b}  \subset B(x,t)\subset B(x_B, 8t)
	\]
	and
	\[
	|L^j \tilde{b}(y)|\lesi \f{t^{2(M-j)}}{V(x,t)}\sim \f{t^{2(M-j)}}{V(x_B,t)}.
	\]
	At this stage, arguing similarly to the proof of the item (ii) in Theorem \ref{thm-spectralmultipliers} below, we can find that 
	\[
	\|K_{(t^2L)^M\Psi^*(t\sL)}(x,\cdot)\|_{\F^{\alpha,L}_{q',q'}(X)}\lesi t^{\alpha}V(Q)^{-1}\sim t^{\alpha}V(x_B,t)^{-1/q}\lesi t^{\alpha}V(B)^{-1/q}
	\]
	Hence,
	\[
	|\Psi(t\sL)a(x)|\lesi \Big(\f{r_B}{t}\Big)^{2M} t^{\alpha}V(B)^{-1/p}.
	\]
	Plugging this into the expression of $I_2$ together with the fact that $\Psi(t\sL)a(x)=0$ as $t<d(x,x_B)/4$, we obtain
	\[
	\begin{aligned}
	I_2^p&\lesi \int_{(8B)^c}\Big\{\int_{d(x,x_B)/4}^\vc  \Big(\f{r_B}{t}\Big)^{2qM}V(B)^{-q/p}\f{dt}{t}\Big\}^{p/q}\dx\\
	&\lesi \int_{(8B)^c} \f{1}{V(B)}\Big(\f{r_B}{d(x,x_B)}\Big)^{2Mp} \dx\\
	&\lesi 1 
	\end{aligned}
	\]
	as long as $M>\f{n}{2p}$.
	
	This completes our proof.
\end{proof}

\begin{rem}
	In fact, in Theorem \ref{thm-new atomic decomposition} we can prove the new atomic decomposition for $f\in \F^{\alpha,L}_{p,q}(X)$ instead of $\F^{\alpha,L}_{p,q}(X)\cap L^2(X)$. To do this, we need the Calder\'on reproducing formula in \cite{BBD} in the new space of distribution $\mathcal{S}'_\vc$. However, we do not pursue this problem.
\end{rem}
\subsection{Duality and interpolations}

The following results regarding the duality and the complex interpolation of the Triebel--Lizorkin follow directly from Proposition \ref{prop-square function characterization} and the duality and complex interpolation results for the weighted tent spaces. See \cite{HHM} for the Euclidean setting and  \cite{Amenta} for the possible extension to the spaces of homogeneous type.
\begin{prop}
	\label{prop-duality}
	Let $\alpha\in \mathbb{R}$ and $1<p,q<\vc$. The dual space $[\F^{\alpha,L}_{p,q}(X)]^*$ of the Triebel--Lizorkin space $\F^{\alpha,L}_{p,q}(X)$ is $\F^{-\alpha,L}_{p',q'}(X)$.
\end{prop}

\begin{prop}
	\label{prop-comple interpolation}
We have
\begin{equation}
\label{complex interpolation}
\left(\F^{\alpha_0,L}_{p_0,q_0}(X),\F^{\alpha_1,L}_{p_1,q_1}(X)\right)_\theta = \F^{\alpha,L}_{p,q}(X)
\end{equation}
for all $\alpha_0, \alpha_1 \in \mathbb{R}$, $0<p_0,p_1, q_0, q_1<\vc$, $\theta\in (0,1)$ and
\[
\alpha=(1-\theta)\alpha_0 +\theta \alpha_1, \ \ \ \f{1}{p} =\f{1-\theta}{p_0}+\f{\theta}{p_1}, \ \ \ \f{1}{q} =\f{1-\theta}{q_0}+\f{\theta}{q_1}
\] 	
where $(\cdot, \cdot)_\theta$ stands for the complex interpolation brackets.
\end{prop}

We now prove a real interpolation result for our new Besov and Triebel--Lizorkin spaces. We first recall the background of real interpolation method in \cite{Tr, BL}. Let $\mathcal{H}$ be a linear complex Hausdorff
space. Assume that $A_1$ and $A_2$ are two complex quasi-Banach spaces such that $A_1\subset \mathcal{H}$ and $A_2\subset \mathcal{H}$. Define $A_1 +A_2=\{a=a_1+a_2: a_i\in A_i, i=1,2\}$. For $0 < t < \infty$ and $a \in A_1 + A_2$, then the $K$-functional is defined by
$$
K(t,a):=K(t,a; A_1, A_2)=\inf\left(\|a_1\|_{A_1}+t\|a_2\|_{A_2}\right),
$$
where the infimum is taken over all representations of $a$ of the form $a = a_1 + a_2$ with $a_i\in A_i, i=1,2$.

\begin{defn}
	\label{def1-RealIn}
	Let $\theta\in (0,1)$. If $0<q\leq \vc$ then
	$$
	(A_1,A_2)_{\theta,q}=\left\{a\in A_1+A_2: \|a\|_{(A_1,A_2)_{\theta,q}}<\vc\right\},
	$$
	where
	$$
	\|a\|_{(A_1,A_2)_{\theta,q}}=\left(\int_0^\vc \left[t^{-\theta}K(t,a)\right]^q\f{dt}{t}\right)^{1/q} \ \text{when $0<q<\vc$},
	$$
	and
	$$
	\|a\|_{(A_1,A_2)_{\theta,q}}=\sup_{t>0} t^{-\theta}K(t,a) \ \text{when $q=\vc$}.
	$$
\end{defn}

We now summarize some basic properties for $(A_1,A_2)_{\theta,q}$ in \cite{Tr, BL}.

\begin{prop}
	\label{prop1-RealIn}
	Let $A_1$ and $A_2$ be two complex quasi-Banach spaces. Let $\theta\in (0,1)$ and $0<q\leq \vc$.The folllowing holds true:
	\begin{enumerate}[{\rm (i)}]
		\item $(A_1,A_2)_{\theta,q}$ is quasi-Banach space;
		\item Let $\mathcal{H}$ be a linear complex Hausdorff
		space. Assume that $B_1$ and $B_2$ are two complex quasi-Banach spaces such that $A_i\subset B_i\subset \mathcal{H}, i=1,2$. Then $(A_1,A_2)_{\theta,q}\subset (B_1,B_2)_{\theta,q}$.
	\end{enumerate}
\end{prop}

Our main result of this section is the following theorem.

\begin{thm}
	\label{mainthm-Interpolation}
	Let $\theta\in (0,1), s_1, s_2\in \mathbb{R}, s_1\neq s_2$ and $s=(1-\theta)s_1+\theta s_2$.
	\begin{enumerate}[{\rm (i)}]
		\item If $0< p,  q_1, q_2, q\leq \vc$ then
		\begin{equation}\label{eq-BesovInter}
		\left(\B^{s_1,L}_{p,q_1}(X),\B^{s_2,L}_{p,q_2}(X)\right)_{\theta,q}=\B^{s,L}_{p,q}(X).
		\end{equation}
		
		\item If $0<p, q_1, q_2, q< \vc$ then
		\begin{equation}\label{eq-TLInter}
		\left(\F^{s_1,L}_{p,q_1}(X),\F^{s_2,L}_{p,q_2}(X)\right)_{\theta,q}=\B^{s,L}_{p,q}(X).
		\end{equation}
		
	\end{enumerate}
\end{thm}
\begin{proof}
	Without loss of generality, we may assume that $s_1<s_2$. We only give the proof for the case $0<q<\vc$, the proof for $q=\vc$ can be done similarly with minor modifications  and we omit the details.

	Let $0<q<\vc$. We will show that 
	\begin{equation}
	\label{eq1-proofInter}
	\left(\B^{s_1,L}_{p,q_1}(X),\B^{s_2,L}_{p,q_2}(X)\right)_{\theta,q}\subset \B^{s,L}_{p,q}(X).
	\end{equation}
	Suppose that $f\in \B^{s,L}_{p,q}(X)$. Then we have
	\begin{equation}
	\label{eq2-proofIn}
	\begin{aligned}
	\|f\|^q_{\left(\B^{s_1,L}_{p,q_1}(X),\B^{s_2,L}_{p,q_2}(X)\right)_{\theta,q}}&=\int_0^\vc \left[t^{-\theta}K(t,f; \B^{s_1,L}_{p,q_1},\B^{s_2,L}_{p,q_2})\right]^q\f{dt}{t}\\
	&=\sum_{k\in \mathbb{Z}}\int_{2^{k(s_1-s_2)}}^{2^{(k+1)(s_1-s_2)}} \left[t^{-\theta}K(t,f; \B^{s_1,L}_{p,q_1},\B^{s_2,L}_{p,q_2})\right]^q\f{dt}{t}\\
	&\sim \sum_{k\in \mathbb{Z}} 2^{-\theta q k (s_1-s_2)}\Big[K\left(2^{k (s_1-s_2)},f; \B^{s_1,L}_{p,q_1}(X),\B^{s_2,L}_{p,q_2}(X)\right)\Big]^q.
	\end{aligned}
	\end{equation}
	Let $\psi$ be a partition of unity. From the definition of the Besov spaces $\B^{\alpha, L}_{p,q}(X)$ we have, for each $k\in \mathbb{Z}$ and each decomposition  $f=f_1+f_2$ with $f_i\in \B^{s_i,L}_{p,q_i}(X), i=1,2$,
	$$
	\begin{aligned}
	\|f_1\|_{\B^{s_1,L}_{p,q_1}}+2^{k (s_1-s_2)}\|f_1\|_{\B^{s_2,L}_{p,q_2}}&\ge 2^{ks_1}\|\psi_k(\sL)f_1\|_p +2^{k (s_1-s_2)}2^{ks_2}\|\psi_k(\sL)f_2\|_p\\
	&=2^{ks_1}\|\psi_k(\sL)f_1\|_p +2^{k s_1 }\|\psi_k(\sL)f_2\|_p\\
	&\gtrsim 2^{ks_1}\|\psi_k(\sL)(f_1+f_2)\|_p =2^{ks_1}\|\psi_k(\sL)f\|_p
	\end{aligned}
	$$
	which implies
	\[
	\Big[K\left(2^{k (s_1-s_2)},f; \B^{s_1,L}_{p,q_1}(X),\B^{s_2,L}_{p,q_2}(X)\right)\Big]^q\gtrsim 2^{ks_1q}\|\psi_k(\sL)f\|^q_p
	\]
	Plugging this into \eqref{eq2-proofIn}, 
	\[
	\begin{aligned}
	\|f\|^q_{\left(\B^{s_1,L}_{p,q_1}(X),\B^{s_2,L}_{p,q_2}(X)\right)_{\theta,q}}
	&\gtrsim \sum_{k\in \mathbb{Z}} 2^{-\theta q k (s_1-s_2)}2^{ks_1q}\|\psi_k(\sL)f\|^q_p\\
	& = \sum_{k\in \mathbb{Z}} \Big[2^{[s_1-\theta  (s_1-s_2)]k}\|\psi_k(\sL)f\|_p\Big]^q\\
	& = \sum_{k\in \mathbb{Z}} \Big[2^{sk}\|\psi_k(\sL)f\|_p\Big]^q=:\|f\|^q_{\B^{s,L}_{p,q}(X)}.
	\end{aligned}
	\]
	This proves \eqref{eq1-proofInter}.
	
	To complete the proof, it suffices to show that
	\begin{equation}
	\label{eq3-proofInter}
	\B^{s,L}_{p,q}(X)\subset \left(\B^{s_1,L}_{p,q_1}(X),\B^{s_2,L}_{p,q_2}(X)\right)_{\theta,q}.
	\end{equation}
	
	Indeed, we temporarily assume  that $q_1\vee q_2\le  q$. For $f\in \B^{s,L}_{p,q}(X)$, by Theorem \ref{thm1- atom Besov}, there exist a sequence of $(L,M,p)$ atoms $\{a_Q\}_{Q\in \mathscr{D}_\nu, \nu\in \mathbb{Z}}$ and a sequence of coefficients  $\{s_Q\}_{Q\in \mathscr{D}_\nu, \nu\in\mathbb{Z}}$ so that
	\begin{equation}\label{eq4-proofInter}
	f=\sum_{\nu\in\mathbb{Z}}\sum_{Q\in \mathscr{D}_\nu}s_Qm_Q \ \ \text{in $\SLvc'$},
	\end{equation}
	and
	\begin{equation}\label{eq5-proofInter}
	\Big[\sum_{\nu\in\mathbb{Z}}2^{\nu sq}\Big(\sum_{Q\in \mathscr{D}_\nu}|s_Q|^p\Big)^{q/p}\Big]^{1/q}\lesi \|f\|_{\B^{s,L}_{p,q}}.
	\end{equation}

	For each $k\in \mathbb{Z}$ we define
	$$
	f_{k,1} =\sum_{\nu\geq k}\sum_{Q\in \mathscr{D}_\nu}s_Qm_Q, \ \ \text{and} 
	\  \ f_{k,2} =\sum_{\nu< k}\sum_{Q\in \mathscr{D}_\nu}s_Qm_Q.
	$$
	Hence, by Theorem \ref{thm2- atom Besov}, we have
	
	$$
	\|f_{k,1}\|_{\B^{s_1,L}_{p,q_1}(X)}\lesi \Big[\sum_{\nu\geq k}2^{\nu s_1q_1}\Big(\sum_{Q\in \mathscr{D}_\nu}|s_Q|^p\Big)^{q_1/p}\Big]^{1/q_1}
	$$
	and
	$$
	\|f_{k,2}\|_{\B^{s_2,L}_{p,q_2}(X)}\lesi\Big[\sum_{\nu< k}2^{\nu s_2q_2}\Big(\sum_{Q\in \mathscr{D}_\nu}|s_Q|^p\Big)^{q_2/p}\Big]^{1/q_2}.
	$$

	Recalling \eqref{eq2-proofIn}, we have
	\begin{equation}
	\label{eq6s-proofIn}
	\begin{aligned}
	\|f\|^q_{\left(\B^{s_1,L}_{p,q_1}(X),\B^{s_2,L}_{p,q_2}(X)\right)_{\theta,q}}&\sim \sum_{k\in \mathbb{Z}} 2^{-\theta q k (s_1-s_2)}\Big[K\left(2^{k (s_1-s_2)},f; \B^{s_1,L}_{p,q_1}(X),\B^{s_2,L}_{p,q_2}(X)\right)\Big]^q.
	\end{aligned}
	\end{equation}

	Therefore,
	$$
	\begin{aligned}
	\sum_{k\in \mathbb{Z}} &2^{-\theta q k (s_1-s_2)}\Big[K\left(2^{k (s_1-s_2)},f; \B^{s_1,L}_{p,q_1}(X),\B^{s_2,L}_{p,q_2}(X)\right)\Big]^q\\
	&\leq \sum_{k\in \mathbb{Z}} 2^{-\theta q k (s_1-s_2)}\left[\|f_{k,1}\|_{\B^{s_1,L}_{p,q_1}(X)} + 2^{k (s_1-s_2)}\|f_{k,2}\|_{\B^{s_2,L}_{p,q_2}(X)}\right]^q\\
	&\leq \sum_{k\in \mathbb{Z}} 2^{-\theta q k (s_1-s_2)}\Big\{\Big[\sum_{\nu\geq k}2^{\nu s_1q_1}\Big(\sum_{Q\in \mathscr{D}_\nu}|s_Q|^p\Big)^{q_1/p}\Big]^{1/q_1}\\
	& \ \ \ \  +2^{k (s_1-s_2)} \Big[\sum_{\nu< k}2^{\nu s_2q_2}\Big(\sum_{Q\in \mathscr{D}_\nu}|s_Q|^p\Big)^{q_2/p}\Big]^{1/q_2} \Big\}^q\\
	&\leq \sum_{k\in \mathbb{Z}}  \Big[\sum_{\nu\geq k}2^{-\theta q_1 k (s_1-s_2)}2^{\nu s_1q_1}\Big(\sum_{Q\in \mathscr{D}_\nu}|s_Q|^p\Big)^{q_1/p}\Big]^{q/q_1}\\
	&\ \ \ \  +\sum_{k\in \mathbb{Z}} \Big[\sum_{\nu< k}2^{k q_2(1-\theta)(s_1-s_2)}2^{\nu s_2q_2}\Big(\sum_{Q\in \mathscr{D}_\nu}|s_Q|^p\Big)^{q_2/p}\Big]^{q/q_2}=: E_1 + E_2.
	\end{aligned}
	$$
	For the term $E_1$, we have
	\[
	\begin{aligned}
	E_1&= \sum_{k\in \mathbb{Z}}  \Big[\sum_{\nu\geq k}2^{\theta q_1(s_1-s_2)(\nu-k)}2^{\nu sq_1}\Big(\sum_{Q\in \mathscr{D}_\nu}|s_Q|^p\Big)^{q_1/p}\Big]^{q/q_1}.
	\end{aligned}
	\]
	Since $\theta q_1(s_1-s_2)(\nu-k)<0$ as $\nu> k$, by H\"older's inequality and Fubini's Theorem we have
	\[
	E_1\lesi \sum_{k\in \mathbb{Z}}  \Big[2^{\nu sq}\Big(\sum_{Q\in \mathscr{D}_\nu}|s_Q|^p\Big)^{q/p}\Big]\lesi \|f\|^q_{\B^{s,L}_{p,q}(X)}.
	\]
	For the same reason, we have
	\[
	\begin{aligned}
	E_2&=\sum_{k\in \mathbb{Z}} \Big[\sum_{\nu< k}2^{q_2(1-\theta)(s_1-s_2)(k-\nu)}2^{\nu sq_2}\Big(\sum_{Q\in \mathscr{D}_\nu}|s_Q|^p\Big)^{q_2/p}\Big]^{q/q_2}\\
	&\lesi \|f\|^q_{\B^{s,L}_{p,q}(X)}.
	\end{aligned}
	\]
	
	As a consequence,
	\[
	\sum_{k\in \mathbb{Z}} 2^{-\theta q k (s_1-s_2)}\Big[K\left(2^{k (s_1-s_2)},f; \B^{s_1,L}_{p,q_1}(X),\B^{s_2,L}_{p,q_2}(X)\right)\Big]^q\lesi  \|f\|^q_{\B^{s,L}_{p,q}(X)}.
	\]
	This, along with \eqref{eq6s-proofIn}, implies \eqref{eq3-proofInter} for the case $q_1\vee q_2\le q$. 
	
	If $q_1>q$ or $q_2>q$, from the embedding of the sequence spaces  $\ell^{q}\hookrightarrow \ell^{q_1\vee q_2}$ we obtain $\B^{s, L}_{p,q}(X)\subset \B^{s,L}_{p,q_1\vee q_2}(X)$ and the proof of \eqref{eq3-proofInter} tells us that 
	\begin{equation}
	\label{eq3-proofInter}
	\B^{s,L}_{p,q_1\vee q_2}(X)\subset \left(\B^{s_1,L}_{p,q_1}(X),\B^{s_2,L}_{p,q_2}(X)\right)_{\theta,q}.
	\end{equation}
	It follows that
	\[
	\B^{s,L}_{p,q}(X)\subset \left(\B^{s_1,L}_{p,q_1}(X),\B^{s_2,L}_{p,q_2}(X)\right)_{\theta,q}
	\]
	since $\B^{s, L}_{p,q}(X)\subset \B^{s,L}_{p,q_1\vee q_2}(X)$.
	
	This completes the proof of (i).
	
	\bigskip
	
	(ii) By Minkowski's inequality  and the embedding $\ell^{p\wedge q}\subset \ell^{q}\subset \ell^{p\vee q}$ we conclude that 
	\begin{equation}
	\label{embedding eq 1}
	\B^{s,L}_{p,q\wedge p}(X) \subset \F^{s,L}_{p,q}(X)\subset \B^{s,L}_{p,p\vee q}(X)
	\end{equation} 
	for all $s\in \mathbb{R}$ and $0<p,q<\vc$.
	
	This, along with Proposition \ref{prop1-RealIn} and the interpolation result in (i), implies
	\[
	\begin{aligned}
	\B^{s,L}_{p,q}(X)&= \left(\B^{s_1,L}_{p,q_1\wedge p}(X),\B^{s_2,L}_{p,q_2\wedge p}(X)\right)_{\theta,q} \subset \left(\F^{s_1,L}_{p,q_1 }(X),\F^{s_2,L}_{p,q_2}(X)\right)_{\theta,q}\\
	&\subset \left(\B^{s_1,L}_{p,q_1\wedge p}(X),\B^{s_2,L}_{p,q_2\wedge p}(X)\right)_{\theta,q}=\B^{s,L}_{p,q}(X).
	\end{aligned}
	\]
	This completes the proof of (ii).
\end{proof}

\section{Spectral multipliers on Besov spaces $\B^{\alpha,L}_{p,q}$ and Triebel--Lizorkin spaces $\F^{\alpha,L}_{p,q}$.}
This section is devoted to the proofs of Theorem \ref{mainthm-spectralmultipliers-Miklin} and Theorem \ref{mainthm-spectralmultipliers}.
We note that the proof of Theorem \ref{mainthm-spectralmultipliers} is quite long and technical, while Theorem \ref{mainthm-spectralmultipliers-Miklin} follows
from Theorem \ref{mainthm-spectralmultipliers}. Indeed, once Theorem \ref{mainthm-spectralmultipliers} is proved, we can obtain Theorem \ref{mainthm-spectralmultipliers-Miklin}
as follows.

\begin{proof}
	[Proof of Theorem \ref{mainthm-spectralmultipliers-Miklin}:] It is known  that \eqref{eq1-mainthmsm} is  true with $\tilde q =\vc$. See Remark 1 in \cite{DOS}. So, Theorem \ref{mainthm-spectralmultipliers-Miklin} is a direct consequence of Theorem \ref{mainthm-spectralmultipliers}.
\end{proof}

We now turn to the proof of Theorem \ref{mainthm-spectralmultipliers}.  We first need  the following result.
\begin{thm}\label{thm-spectralmultipliers} 		Let $s>\f{n}{2}$ and let   $\alpha\in \mathbb{R}$ and $0<p,q<\vc$. Assume that  the condition \eqref{eq1-mainthmsm}  holds true. If $F$ is a  bounded Borel function which satisfies the  condition \eqref{smoothness condition}, then
	\begin{enumerate}[(a)]
		\item the spectral multiplier $F(\sqrt L)$ is bounded on $\F^{\alpha, L}_{p,2}(X)$ provided that $\alpha\in \mathbb{R}$, $0<p<\vc$ and $s>n(\f{1}{p\wedge 1}-\f{1}{2})$; moreover,
		\begin{equation}\label{eq1-thm1}
							\|F(\sqrt L)\|_{\F^{\alpha, L}_{p,2}(X)\to \F^{\alpha, L}_{p,2}(X)}\lesi  \|F\|_{L^\vc} +\sup_{t>0}\|\eta\, \delta_tF\|_{W^{\tilde q}_s}.
		\end{equation}
		\item the spectral multiplier $F(\sqrt L)$ is bounded on $\F^{\alpha, L}_{p,q}(X)$ provided that $\alpha\in \mathbb{R}$, $0<p,q< \vc$ and $s>\f{n}{1\wedge p\wedge q}+\f{\tilde{n}}{2}$; moreover,
				\begin{equation}\label{eq2-thm1}
				\|F(\sqrt L)\|_{\F^{\alpha, L}_{p,q}(X)\to \F^{\alpha, L}_{p,q}(X)}\lesi \|F\|_{L^\vc}+ \sup_{t>0}\|\eta\, \delta_tF\|_{W^{\tilde q}_s}.
		\end{equation}
	\end{enumerate}
\end{thm}
\bigskip

We note that the smoothness condition $s>n(\f{1}{p\wedge 1}-\f{1}{2})+\f{\tilde{n}}{2}$ in Theorem \ref{thm-spectralmultipliers} (b) is not as sharp as we expect, i.e. $s>n(\f{1}{1\wedge p \wedge q}-\f{1}{2})$. However, we will be able to obtain the sharp estimate for $s$ by an interpolation argument.

\medskip

Before coming to the proof of Theorem \ref{thm-spectralmultipliers}, we prove the following technical lemmas regarding the kernel estimates for spectral multipliers with compact supports.
\begin{lem}\label{lem2-sm}
	Assume that $L$ satisfies (\ref{eq1-mainthmsm}). Then for any $x,y\in X$,
	\begin{equation}
	\label{eq-lem2-sm}
	|K_{F(\sL)}(x, y)|\leq \f{C}{\sqrt{V(x,R^{-1})V(y,R^{-1})}}\|\delta_R F\|_{\tilde q}
	\end{equation}
	for all bounded Borel function $F$ with ${\rm supp}\, \subset [0,R]$.
\end{lem}

\begin{proof} We write $F(\lambda)=e^{-\f{\lambda^2}{R^2}}G(\lambda)$ and hence $G(\lambda)=e^{\f{\lambda^2}{R^2}}F(\lambda)$. Since supp\, $F\subset [0,R]$, supp\, $G\subset [0, R]$ and in addition $\|\delta_R F\|_{\tilde q}\approx \|\delta_R G\|_{\tilde q}$. Besides, $F(\sL) = e^{-\f{1}{R^2}}G(\sqrt{\lambda})$ which implies that
	$$
	K_{F(\sL)}(x, y)= \int_X p_{R^{-2}}(x,z)K_{G(\sL)}(z,y)dz.
	$$
	Using H\"older's inequality,
	\[
	|K_{F(\sL)}(x, y)|\le \Big[\int_X|p_{R^{-2}}(x,z)|^2dz\Big]^{1/2}\Big[ \int_X|K_{G(\sL)}(z,y)|^2dz\Big]^{1/2}
	\]
	Using  \eqref{eq1-mainthmsm} and Lemma \ref{lem-elementary}, we obtain
	$$
	\begin{aligned}
	|K_{F(\sL)}(x, y)|&\lesi \f{\|\delta_R G\|_{\tilde q}}{V(x,R^{-1})^{\f{1}{2}}V(y,R^{-1})^{\f{1}{2}}}\\
	&\sim \f{\|\delta_R F\|_{\tilde q}}{V(x,R^{-1})^{\f{1}{2}}V(y,R^{-1})^{\f{1}{2}}}.
	\end{aligned}
	$$
	This completes our proof.
\end{proof}

\begin{lem}\label{lem4-sm}
	Let $R,s>0$. Then for any $\epsilon>0$ there exists a constant $C=C(s,\epsilon)$ such that
	\begin{equation}\label{eq1-lem4}
	|K_{F(\sL)}(x,y)| \leq  C\f{(1+Rd(x,y))^{-s}}{\sqrt{V(x,R^{-1})V(y,R^{-1})}}\|\delta_RF\|_{W^{\tilde q}_{s+\epsilon}}
	\end{equation}
	for any bounded Borel function $F$ supported in $[R/4, R]$ and for all $x,y\in X$.
\end{lem}
\begin{proof} By the Fourier inversion formula
	$$
	G(L/R^2)e^{-L/R^2}=\f{1}{2\pi}\int_\mathbb{R} \exp((i\tau -1)R^{-2}L)\widehat{G}(\tau)d\tau
	$$
	and so
	$$
	K_{F(\sL)}(x,y)=\f{1}{2\pi}\int_\mathbb{R}\widehat{G}(\tau)p_{(1-i\tau)/R^2}(x,y)d\tau
	$$
	where $G(\lambda)=[\delta_RF](\sqrt{\lambda})e^{\lambda}$.
	
	We note that from the Gaussian upper bound and the doubling condition \eqref{doubling1} we have
	\[
	|p_{(1-i\tau)/R^2}(x,y)|\le \f{C}{\sqrt{V(x,R^{-1})V(y,R^{-1})}}\exp\Big(-\f{R^2d(x,y)^2}{(1+\tau^2)}\Big).
	\]
	See \cite{O}.

	Therefore, 
	\begin{equation*}
	\begin{aligned}
	|K_{F(\sL)}(x,y)|
	&\leq C \f{(1+Rd(x,y))^{-s}}{\sqrt{V(x,R^{-1})V(y,R^{-1})}}\int_\mathbb{R}|\widehat{G}(\tau)|(1+|\tau|)^sd\tau\\
	&\leq C\f{(1+Rd(x,y))^{-s}}{\sqrt{V(x,R^{-1})V(y,R^{-1})}}\Big(\int_\mathbb{R}|\widehat{G}(\tau)|^2(1+|\tau|^2)^{s+\epsilon+1/2}d\tau\Big)^{1\over2}\\
	&~~~~~~~~\Big(\int_\mathbb{R}(1+|\tau|^2)^{-\epsilon-1/2}d\tau\Big)^{1\over 2}\\
	&\leq C\f{(1+Rd(x,y))^{-s}}{\sqrt{V(x,R^{-1})V(y,R^{-1})}}\|G\|_{W^2_{s+\epsilon+1/2}}.
	\end{aligned}
	\end{equation*}
	Since supp\, $F\subset [R/4, R]$, $\|G\|_{W^2_{s+\epsilon+1/2}}\leq C\|\delta_RF\|_{W^2_{s+\epsilon+1/2}}\leq C\|\delta_RF\|_{W^{\tilde q}_{s+\epsilon+1/2}}$. Hence,
	\begin{equation}\label{eq2-lem4}
	|K_{F(\sL)}(x,y)|\leq C
	\f{(1+Rd(x,y))^{-s}}{\sqrt{V(x,R^{-1})V(y,R^{-1})}}\|\delta_RF\|_{W^{\tilde q}_{s+\epsilon+1/2}}.
	\end{equation}
	Interpolating \eqref{eq2-lem4} and \eqref{eq-lem2-sm}, we obtain \eqref{eq1-lem4} as desired.
	
	This completes our proof.
	
\end{proof}

\bigskip

We now recall an estimate in \cite[Lemma 4.3]{DOS}.
\begin{lem}\label{lem-Lp sm}
	Let $R,s>0$. Then for any $\epsilon>0$  there exists a constant $C=C(s,\epsilon)$ such that
	\begin{equation}\label{eq-lem Lp sm}
	\Big[\int_X|K_{F(\sL)}(x,y)|^2(1+Rd(x,y))^{2s}  \dy\Big]^{1/2} \leq  \f{C}{V(x,R^{-1})^{\f{1}{2}}}\|\delta_RF\|_{W^{\tilde q}_{s+\epsilon}}
	\end{equation}
	for any bounded Borel function $F$ supported in $[R/4, R]$ and for all $x,y\in X$.
\end{lem}

We now prove Theorem \ref{thm-spectralmultipliers}

\begin{proof}
	[Proof of (i) of Theorem \ref{thm-spectralmultipliers}:] 
	Suppose that $\psi$ is a partition of unity. Let $\varphi\in \mathscr{\mathbb{R}}$ such that $\supp \varphi\subset [1/4,4]$ and $\varphi=1$ on $[1/2,2]$. Then for each $t>0$, $x\in X$ and $f\in \F^{\alpha, L}_{p,2}(X)\cap L^2(X)$ with $p\ge 2$, we have
	\[
	\begin{aligned}
	t^{-\alpha}|\psi(t\sL)F(\sL)f(x)| &= t^{-\alpha}|\varphi(t\sL)\psi(t\sL)F(\sL)f(x)|=t^{-\alpha}|\varphi(t\sL)F(\sL)[\psi(t\sL)f](x)|\\
	&=\Big|\int_X t^{-\alpha}K_{\varphi(t\sL)F(\sL)}(x,y)\psi(t\sL)f(y)\dy\Big|\\
	&\le \Big[\int_X |K_{\varphi(t\sL)F(\sL)}(x,y)|^2 \Big(1+\f{d(x,y)}{t}\Big)^{2s'} \dy\Big]^{1/2}\\
	& \ \ \ \ \times \Big[\int_X   \Big(1+\f{d(x,y)}{t}\Big)^{-2s'}(t^{-\alpha}|\psi(t\sL)f(y)|)^2 dy\Big]^{1/2}
	\end{aligned}
	\]
	where $s>s'>\f{n}{2}$.
	
	Applying Lemma \ref{lem-Lp sm},
	\[
	\begin{aligned}
	\Big[\int_X |K_{\varphi(t\sL)F(\sL)}(x,y)|^2 \Big(1+\f{d(x,y)}{t}\Big)^{2s'} \dy\Big]^{1/2}&\lesi \f{1}{V(x,t)^{\f{1}{2}}}\|\varphi \delta_{t^{-1}}F\|_{W^{\tilde q}_s}\\
	&\lesi \f{1}{V(x,t)^{\f{1}{2}}}.
	\end{aligned}
	\]
	As a consequence,
	\[
	\begin{aligned}
t^{-\alpha}|\psi(t\sL)F(\sL)f(x)| &\lesi \Big[\int_X   \Big(1+\f{d(x,y)}{t}\Big)^{-2s'}(t^{-\alpha}|\psi(t\sL)f(y)|)^2 \f{dy}{V(x,t)}\Big]^{1/2}.
	\end{aligned}
	\]
	It follows
	\[
	\Big(\int_0^\vc (t^{-\alpha}|\psi(t\sL)F(\sL)f(x)|)^2\f{dt}{t}\Big)^{1/2}\lesi \mathcal G_{s', 2} (\psi(t\sL)f)(x).
	\]
	Therefore, by  Proposition \ref{prop-square function characterization} we have
	\[
	\begin{aligned}
	\|F(\sL)f\|_{\F^{\alpha, L}_{p,2}(X)}&\sim \Big\|\Big(\int_0^\vc |\psi(t\sL)F(\sL)f|^2\f{dt}{t}\Big)^{1/2}\Big\|_{p}\\
	&\lesi \left\|\mathcal G_{s', 2} (\psi(t\sL)f)\right\|_{p} \sim \|f\|_{\F^{\alpha, L}_{p,2}(X)}.
	\end{aligned}
	\]
    This implies that $F(\sL)$ is bounded on $\F^{\alpha, L}_{p,2}(X)$ for all $\alpha\in \mathbb{R}$ and $p\ge 2$. By the duality result in Proposition \ref{prop-duality}, $F(\sL)$ is bounded on $\F^{\alpha, L}_{p,2}(X)$ for all $\alpha\in \mathbb{R}$ and $1<p<\vc$.
    
    \medskip
    
    It remains to prove that $F(\sL)$ is bounded on $\F^{\alpha, L}_{p,2}(X)$ for all $\alpha\in \mathbb R$, $0<p\le 1$ and $s>n(\f{1}{p}-\f{1}{2})$. 
     From Theorem \ref{thm-new atomic decomposition}, it suffices to prove that there exists $C>0$ so that
    \begin{equation}\label{eq1-proof FL}
    \|F(\sL)a\|_{\F^{\alpha, L}_{p,2}(X)}\le C
    \end{equation}
    for each $(L,M,\alpha, p,q)$ atom associated to some ball $B\subset X$.
    
    Let $\Phi$ be a function as in Lemma \ref{lem:finite propagation}. As in the proof of \eqref{eq1-proof new atomic decomposition}, there exists $\epsilon>0$ such that $\varphi(\xi):=|\xi|^{2m}\Phi(\xi)\ne 0$ on $\{\xi: \epsilon/2\le |\xi|\le 2\epsilon\}$ for $m>\alpha/2$. By Proposition \ref{prop4.1-heat kernel} and Remark \ref{rem0}, it suffices to prove that 
    \begin{equation}\label{eq2-proof FL}
    \Big\|\Big(\int_0^\vc(t^{-\alpha}|\psi(\epsilon t\sL)\varphi(t\sL)F(\sL)a|)^2\f{dt}{t}\Big)^{1/2}\Big\|_{p}\le C
    \end{equation}
    for each $(L,M,\alpha, p,2)$ atom $a$ associated to some ball $B\subset X$.
    
We note, by using Lemma \ref{lem:finite propagation}, that
    \begin{equation}\label{eq-finite propagation epsilon=1}
    K_{\varphi(t\sL)}(\cdot,\cdot)\subset  \{(x,y)\in X\times X:
    d(x,y)\leq t\}.
    \end{equation}

    To do this, we write
    \begin{equation}
    \label{eq-Fp2 norm}
    \begin{aligned}
    \Big\|\Big(\int_0^\vc (t^{-\alpha}&|\psi(\epsilon t\sL)\varphi(t\sL)F(\sL)a|)^2\f{dt}{t}\Big)^{1/2}\Big\|_{p}\\
    &\lesi \Big\|\Big(\int_0^{r_B} (t^{-\alpha}|\psi(\epsilon t\sL)\varphi(\epsilon t\sL)F(\sL)a|)^2\f{dt}{t}\Big)^{1/2}\Big\|_{p}\\
    & \ \ \  +\Big\|\Big(\int^\vc_{r_B} (t^{-\alpha}|\psi(\epsilon t\sL)\varphi(t\sL)F(\sL)a|)^2\f{dt}{t}\Big)^{1/2}\Big\|_{p}=:E_1 + E_2.
    \end{aligned}
    \end{equation}
    Let us take care of $E_1$ first. Observe that 
    \[
    \begin{aligned}
    E_1^p = \sum_{j\ge 0}\Big\|\Big(\int_0^{r_B} (t^{-\alpha}|\psi(\epsilon t\sL)\varphi( t\sL)F(\sL)a|)^2\f{dt}{t}\Big)^{1/2}\Big\|^p_{L^p(S_j(B))}=:\sum_{j\ge 0} E_{1j}.
    \end{aligned}
    \]
    For $j=0,1,2,3$, by H\"older's inequality we have
    \[
    \begin{aligned}
    E_{1j}&\lesi \Big\|\Big(\int_0^\vc(t^{-\alpha}|\psi(\epsilon t\sL)\varphi(t\sL)F(\sL)a|)^2\f{dt}{t}\Big)^{1/2}\Big\|^p_{2} V(2^jB)^{1-\f{p}{2}}\\
    &\lesi  \Big(\int_0^\vc (t^{-\alpha}\|\psi(\epsilon t\sL)\varphi(t\sL)F(\sL)a\|_2)^2\f{dt}{t}\Big)^{p/2}V(B)^{1-\f{p}{2}}.
    \end{aligned}
    \]
    Since $\varphi(t\sL)F(\sL)$ is bounded on $L^2(X)$, we have $\|\psi(\epsilon t\sL)\varphi(t\sL)F(\sL)a\|_2\lesi \|F\|_{L^\vc}\|\psi(\epsilon t\sL)a\|_2$. Hence,
    \[
    \begin{aligned}
    E_{1j}	&\lesi  \Big(\int_0^\vc (t^{-\alpha}\|\psi(\epsilon t\sL)a\|_2)^2\f{dt}{t}\Big)^{p/2}V(B)^{1-\f{p}{2}}\\
    &\lesi  \Big(\int_0^\vc (t^{-\alpha}\|\psi( \epsilon t\sL)a\|_2)^2\f{dt}{t}\Big)^{p/2}V(B)^{1-\f{p}{2}}\\
    &\sim \|a\|^p_{\F^{\alpha,L}_{2,2}}V(B)^{1-\f{p}{2}} \\
    &\lesi 1.
    \end{aligned}
    \]
    For $j> 3$ we have
    \[
    \begin{aligned}
    E_{1j}\le \int_{S_j(B)}\Big[\int_0^{r_B}\Big(\int_X |K_{\psi(\epsilon t\sL)F(\sL)}(x,y)||\varphi(t\sL)a(y)|\dy\Big)^2\f{dt}{t^{1+2\alpha}}\Big]^{p/2}\dx.
    \end{aligned}
    \]
    Due to \eqref{eq-finite propagation epsilon=1}, $\supp\, \varphi(t\sL)a\subset 2B$ as $t<r_B$. Hence,
    \begin{equation}\label{eq1-E1j}
    \begin{aligned}
    E_{1j}&\le \int_{S_j(B)}\Big[\int_0^{r_B}\Big(\int_{2B} |K_{\psi(\epsilon  t\sL)F(\sL)}(x,y)||\varphi(t\sL)a(y)|\dy\Big)^2\f{dt}{t^{1+2\alpha}}\Big]^{p/2}\dx\\
    &\le \Big[\int_{S_j(B)}\int_0^{r_B}\Big(\int_{2B} |K_{\psi(\epsilon t\sL)F(\sL)}(x,y)||\varphi(t\sL)a(y)|\dy\Big)^2\f{dt}{t^{1+2\alpha}}\dx\Big]^{p/2}V(2^jB)^{1-\f{p}{2}}
    \end{aligned}
    \end{equation}
    where in the last inequality we used H\"older's inequality.
    
    From Lemma \ref{lem-Lp sm} we have, for each $y\in 2B$,
    \begin{equation*}
    \begin{aligned}
    \Big[\int_{S_j(B)}&|K_{\psi(\epsilon  t\sL)F(\sL)}(x,y)|^2\dx \Big]^{1/2}\\
    &\lesi \Big(\f{\epsilon  t}{2^jr_B}\Big)^{s'} \Big[\int_{S_j(B)}|K_{\varphi(t\sL)F(\sL)}(x,y)|^2\Big(1+\f{d(x,y)}{t}\Big)^{2s'} \dx \Big]^{1/2}\|\psi \delta_{\epsilon t} F\|_{W^{\tilde q}_{s}}\\
    &\lesi \Big(\f{t}{2^jr_B}\Big)^{s'} \f{1}{V(y,t)^{\f{1}{2}}}\|\psi \delta_{\epsilon t} F\|_{W^{\tilde q}_{s}}.
    \end{aligned}
    \end{equation*}
    for $s>s'>n(\f{1}{p}-\f{1}{2})$.
    
    From the doubling conditions \eqref{doubling1}-\eqref{doubling3},
    \[
    \f{1}{V(y,t)^{\f{1}{2}}}\lesi \Big(\f{2^jr_B}{t}\Big)^{\f{n}{2}}\f{1}{V(y,2^jr_B)^{\f{1}{2}}}\sim \Big(\f{2^jr_B}{t}\Big)^{\f{n}{2}}\f{1}{V(2^jB)^{\f{1}{2}}}.
    \]
    Therefore,
    \begin{equation}\label{eq2-E1j}
    \begin{aligned}
    \Big[\int_{S_j(B)}|K_{\varphi(t\sL)F(\sL)}(x,y)|^2\dx \Big]^{1/2}&\lesi \Big(\f{t}{2^jr_B}\Big)^{s'-\f{n}{2} } \f{1}{V(2^jB)^{\f{1}{2}}}.
    \end{aligned}
    \end{equation}
    We now apply Minkowski's inequality for \eqref{eq1-E1j} and use \eqref{eq2-E1j} to obtain
    \begin{equation*}
    \begin{aligned}
    E_{1j}&\lesi \Big\{\int_0^{r_B}\Big[\int_{2B} \Big(\int_{S_j(B)}|K_{\psi(\epsilon t\sL)F(\sL)}(x,y)|^2\dx\Big)^{1/2}|\varphi(t\sL)a(y)|\dy\Big]^2\f{dt}{t^{1+2\alpha}}\Big\}^{p/2}V(2^jB)^{1-\f{p}{2}}\\
    &\lesi \Big\{\int_0^{r_B}\Big(\f{\epsilon t}{2^jr_B}\Big)^{2s'-n } \f{1}{V(2^jB)}\Big[\int_{2B} |\varphi( t\sL)a(y)|\dy\Big]^2\f{dt}{t^{1+2\alpha}}\dx\Big\}^{p/2}V(2^jB)^{1-\f{p}{2}}.
    \end{aligned}
    \end{equation*}
    Using H\"older's inequality,
    \begin{equation*}
    \begin{aligned}
    E_{1j}
    &\lesi \Big\{\int_0^{r_B}\Big(\f{t}{2^jr_B}\Big)^{2s'-n } \f{V(2B)}{V(2^jB)}\int_{2B} |\varphi(t\sL)a(y)|^2\dy \f{dt}{t^{1+2\alpha}}\dx\Big\}^{p/2}V(2^jB)^{1-\f{p}{2}}\\
    &\lesi 2^{-jp(s'-\f{n}{2})}\Big\{\int_0^{r_B}\int_{2B} |\varphi(t\sL)a(y)|^2\dy \f{dt}{t^{1+2\alpha}}\dx\Big\}^{p/2}V(2^jB)^{1-p}V(B)^{p/2}\\
    &\sim 2^{-jp(s'-\f{n}{2})}  \|a\|^p_{\F^{\alpha, L}_{p,2}(X)}V(2^jB)^{1-p}V(B)^{p/2}\\
    &\lesi 2^{-jp(s'-\f{n}{2})}V(2^jB)^{1-p}V(B)^{p-1}.
    \end{aligned}
    \end{equation*}
    This, along with the doubling condition \eqref{doubling2}, implies
    \[
    E_{1j}\lesi 2^{-jp(s'-\f{n}{2}) + jn(1-p)}= 2^{-jp\left[s'-n(\f{1}{p}-\f{1}{2})\right]}.
    \]
    As a consequence,
    \[
    E_1^p=\sum_{j\ge 0} E_{1j} \lesi 1
    \]
    as along as $s'>n(\f{1}{p}-\f{1}{2})$.

    We now estimate the term $E_2$. Similarly to $E_1$, we write  
    \[
    \begin{aligned}
    E_2^p = \sum_{j\ge 0}\Big\|\Big(\int_{r_B}^\vc (t^{-\alpha}|\psi(\epsilon t\sL)\varphi(t\sL)F(\sL)a|)^2\f{dt}{t}\Big)^{1/2}\Big\|^p_{L^p(S_j(B))}=:\sum_j E_{2j}.
    \end{aligned}
    \]
    Arguing similarly to the terms $E_{2j}$, for $j=0,1,2,3$, we have
    \[
    \begin{aligned}
    E_{2j}&\lesi 1.
    \end{aligned}
    \]
    For $j> 3$ we have
    \[
    \begin{aligned}
    E_{2j}\le \int_{S_j(B)}\Big[\int_0^{r_B}\Big(\int_X |K_{\psi(\epsilon t\sL)F(\sL)}(x,y)||\varphi(t\sL)a(y)|\dy\Big)^2\f{dt}{t^{1+2\alpha}}\Big]^{p/2}\dx.
    \end{aligned}
    \]
    Due to \eqref{eq-finite propagation epsilon=1}, $\supp\, L^M\varphi(t\sL)b\subset B(x_B, 2t)$ as $t\ge r_B$ where $a=L^Mb$. Hence,
    \begin{equation}\label{eq1-E2j}
    \begin{aligned}
    E_{2j}&\le \int_{S_j(B)}\Big[\int^\vc_{r_B}\Big(\int_{B(x_B, 2t)} |K_{\psi(\epsilon t\sL)F(\sL)}(x,y)||L^M\varphi(t\sL)b(y)|\dy\Big)^2\f{dt}{t^{1+2\alpha}}\Big]^{p/2}\dx\\
    &\le \Big[\int_{S_j(B)}\int_{r_B}^\vc \Big(\int_{B(x_B, 2t)} |K_{\psi(\epsilon t\sL)F(\sL)}(x,y)||L^M\varphi(t\sL)b(y)|\dy\Big)^2\f{dt}{t^{1+2\alpha}}\dx\Big]^{p/2}V(2^jB)^{1-\f{p}{2}}.
    \end{aligned}
    \end{equation}
    where in the last inequality we used H\"older's inequality.
    
    Using the argument as above, we obtain
    \begin{equation*}
    \begin{aligned}
    E_{2j}&\lesi \Big\{\int_{r_B}^\vc\Big(\f{t}{2^jr_B}\Big)^{2s'-n } \f{1}{V(2^jB)}\Big[\int_{B(x_B,2t)} |(t^2L)^M\varphi(t\sL)b(y)|\dy\Big]^2\f{dt}{t^{1+2\alpha+4M}}\dx\Big\}^{p/2}V(2^jB)^{1-\f{p}{2}}.
    \end{aligned}
    \end{equation*}
    By H\"older's inequality,
    \begin{equation*}
    \begin{aligned}
    E_{2j}
    &\lesi \Big\{\int^\vc_{r_B}\Big(\f{t}{2^jr_B}\Big)^{2s'-n } \f{V(x_B,2t)}{V(2^jB)}\int_{B(x_B,2t)} |(t^2L)^M\varphi(t\sL)b(y)|^2\dy \f{dt}{t^{1+2\alpha+4M}}\dx\Big\}^{p/2}V(2^jB)^{1-\f{p}{2}}.
    \end{aligned}
    \end{equation*}
    Taking $M> \f{n}{4}(2s'-n)+n$, we have
    \begin{equation*}
    \begin{aligned}
    E_{2j}&\lesi 2^{-jp(s'-\f{n}{2})}r_B^{-2Mp}\Big\{\int_0^{r_B}\int_{B(x_B,2t)} |\varphi(t\sL)a(y)|^2\dy \f{dt}{t^{1+2\alpha}}\dx\Big\}^{p/2}V(2^jB)^{1-p}V(B)^{p/2}\\
    &\sim 2^{-jp(s'-\f{n}{2})}r_B^{-2Mp}  \|b\|^p_{\F^{\alpha, L}_{p,2}(X)}V(2^jB)^{1-p}V(B)^{p/2}\\
    &\lesi 2^{-jp(s'-\f{n}{2})}V(2^jB)^{1-p}V(B)^{p-1}\\
    &\lesi 2^{-jp(s'-\f{n}{2}) + jn(1-p)}= 2^{-jp\left[s'-n(\f{1}{p}-\f{1}{2})\right]}.
    \end{aligned}
    \end{equation*}
    
    It follows that
    \[
    E_2^p=\sum_{j\ge 0} E_{2j} \lesi 1
    \]
    as along as $s'>n(\f{1}{p}-\f{1}{2})$.
    
    The estimates of $E_1$, $E_2$ and \eqref{eq-Fp2 norm} imply that \eqref{eq2-proof FL} as desired. This completes the proof of the item (i) in Theorem \ref{thm-spectralmultipliers}.
    
\end{proof}	
	
	\bigskip

The proof of the item (ii) in Theorem \ref{thm-spectralmultipliers} relies on the following lemma:
	
	\begin{lem}\label{lem2-thm2 atom Besov}
		Let $\psi$ be a partition of unity and  let $a_Q$ be an $(L, M, p)$ atom with some $Q\in \mathscr{D}_\nu$. Then for any $t>0$ and  $N>0$ we have:
		\begin{equation}\label{eq- psi atom}
		|\psi(t\sL)F(\sL) a_Q(x)|\lesi \Big(\f{t}{2^{-\nu}}\wedge \f{2^{-\nu}}{t}\Big)^{2M-n}V(Q)^{-1/p} \Big(1+\f{d(x,x_Q)}{2^{-\nu}\vee t}\Big)^{-s'+\f{\tilde n}{2}}
		\end{equation}
		for any $s>s'>\f{n}{1\wedge p\wedge q}+\f{\tilde{n}}{2}$.
	\end{lem}
	\begin{proof}
		Note that from \eqref{eq1-lem4} and \eqref{doubling3}, we have,  	
		\begin{equation}\label{eq1s-lem4}
		\begin{aligned}
		|K_{\psi(t\sL)F(\sL)}(x,y)| &\lesi   \f{1}{V(x\vee y, t)}\Big(1+\f{d(x,y)}{t}\Big)^{-s'+\f{\tilde n}{2}}\|\psi\delta_{t^{-1}}F\|_{W^{\tilde q}_{s}}\\
		&\lesi   \f{1}{V(x\vee y, t)}\Big(1+\f{d(x,y)}{t}\Big)^{-s'+\f{\tilde n}{2}}
		\end{aligned}
		\end{equation}
		for all $t>0$ and all $x,y\in X$ where $s'\in (0,s)$ such that $s'>\f{\tilde n}{2}+\f{ n}{1\wedge p\wedge q}$.
		
		We now consider two cases: $t\le 2^{-\nu}$ and $t>2^{-\nu}$.

		\noindent{\bf Case 1: $t\le 2^{-\nu}$.} Observe that
		$$
		\psi(t\sL)F(\sL) a_Q=t^{2M}\psi_M(t\sL)F(\sL)(L^Ma_Q) 
		$$
		where $\psi_M(\lambda)=\lambda^{-2M}\psi(\lambda)$.
		
		This, along with \eqref{eq1s-lem4} and the definition of the atoms, yields
		$$
		\begin{aligned}
		|\psi(t\sL) a_Q(x)|&\lesi \int_{3B_Q} \f{t^{2M}}{V(y,t)}\Big(1+\f{d(x,y)}{t}\Big)^{-N}|L^{M}a_Q(y)|d\mu(y)\\
		&\lesi \Big(\f{t}{2^{-\nu}}\Big)^{2M}V(Q)^{-1/p} \int_{3B_Q}\f{1}{V(y,t)}\Big(1+\f{d(x,y)}{t}\Big)^{-s'+\f{\tilde n}{2}}d\mu(y).
		\end{aligned}
		$$
		Note that, for $t\le 2^{-\nu}$ and $y\in 3B_Q$, we have
		\[
		\Big(1+\f{d(x,y)}{t}\Big)^{-s'+\f{\tilde n}{2}}\le \Big(1+\f{d(x,y)}{2^{-\nu}}\Big)^{-s'+\f{\tilde n}{2}}\sim \Big(1+\f{d(x,x_Q)}{2^{-\nu}}\Big)^{-s'+\f{\tilde n}{2}}.
		\]
		Therefore,
		$$
		\begin{aligned}
		|\psi(t\sL)F(\sL) a_Q(x)|&\lesi \Big(\f{t}{2^{-\nu}}\Big)^{2M}V(Q)^{-1/p}\Big(1+\f{d(x,x_Q)}{2^{-\nu}}\Big)^{-s'+\f{\tilde n}{2}} \f{V(3B_Q)}{V(y,t)}\\
		&\lesi \Big(\f{t}{2^{-\nu}}\Big)^{2M-n}V(Q)^{-1/p}\Big(1+\f{d(x,x_Q)}{2^{-\nu}}\Big)^{-s'+\f{\tilde n}{2}}
		\end{aligned}
		$$
		where in the last inequality we use \eqref{doubling2}. This leads us to \eqref{eq- psi atom}.
		
		\medskip
		
		\noindent{\bf Case 2: $t> 2^{-\nu}$.} We first write  $a_Q=L^{M}b_Q$. Hence,
		\[
		\psi(t\sL)F(\sL)a_Q=t^{-2M}\tilde \psi_M(t\sL)F(\sL)b_Q
		\]
		where $\tilde \psi_M(\lambda)=\lambda^{2M}\psi(\lambda)$.
		
		This, along with Lemma \ref{lem1}, implies 
		$$
		\begin{aligned}
		|\psi(t\sL)F(\sL)a_Q(x)|&\lesi \int_{3B_Q}\f{t^{-2M}}{V(y,t)}\Big(1+\f{d(x,y)}{t}\Big)^{-s'+\f{\tilde n}{2}}|b_Q(y)|d\mu(y)\\
		&\lesi \Big(\f{2^{-\nu}}{t}\Big)^{2M} V(Q)^{-1/p}\int_{3B_Q}\f{1}{V(y,t)}\Big(1+\f{d(x,y)}{t}\Big)^{-s'+\f{\tilde n}{2}}d\mu(y).
		\end{aligned}
		$$
		Note that for $y\in 3B_Q$ and $t\ge 2^{-\nu}\sim \ell(Q)$ we have
		\[
		\Big(1+\f{d(x,y)}{t}\Big)^{-s'+\f{\tilde n}{2}}\sim \Big(1+\f{d(x,x_Q)}{t}\Big)^{-s'+\f{\tilde n}{2}}. 
		\]
		Hence, the above inequality simplifies into 
		$$
		\begin{aligned}
		|\psi(t\sL)F(\sL)a_Q(x)|
		&\lesi \Big(\f{2^{-\nu}}{t}\Big)^{2M} V(Q)^{-1/p}\Big(1+\f{d(x,x_Q)}{t}\Big)^{-s'+\f{\tilde n}{2}} \f{V(3B_Q)}{V(y,t)}\\
		&\lesi \Big(\f{2^{-\nu}}{t}\Big)^{2M} V(Q)^{-1/p}\Big(1+\f{d(x,x_Q)}{t}\Big)^{-s'+\f{\tilde n}{2}}.
		\end{aligned}
		$$
		The desired estimate  \eqref{eq- psi atom} then follows.
	\end{proof}

The following lemma is taken from \cite{BBD}.
\begin{lem}\label{lem1- thm2 atom Besov}
	Let $N>n$ and $\eta, \nu \in \mathbb{Z}$, $\nu\geq \eta$. Assume that  $\{f_Q\}_{Q\in \mathscr{D}_\nu}$ is a  sequence of functions satisfying
	$$
	|f_Q(x)|\lesi  \Big(1+\f{d(x,x_Q)}{2^{-\eta}}\Big)^{-N}.
	$$
	Then for $\f{n}{N}<r\leq 1$ and a sequence of numbers $\{s_Q\}_{Q\in \mathscr{D}_\nu}$, we have
	$$
	\sum_{Q\in \mathscr{D}_\nu}|s_Q|\,|f_Q(x)|\lesi 2^{n(\nu-\eta)/r}\mathcal{M}_{r}\Big(\sum_{Q\in \mathscr{D}_\nu}|s_Q|\chi_Q\Big)(x).
	$$
\end{lem}

We now give the proof for the item (ii) in Theorem \ref{thm-spectralmultipliers}.
\begin{proof}
	[Proof of (ii) of Theorem \ref{thm-spectralmultipliers}:] Fix $s'$ such that 
	$s>s'>\f{n}{1\wedge p\wedge q}+\f{\tilde{n}}{2}$.
	 
Let  $f\in \F^{\alpha, L}_{p,q}(X)\cap L^2(X)$. By Theorem \ref{thm1- atom TL spaces},  there exist a sequence of $(L,M,p)$ atoms $\{a_Q\}_{Q\in \mathscr{D}_\nu, \nu\in \mathbb{Z}}$ and a sequence of coefficients  $\{s_Q\}_{Q\in \mathscr{D}_\nu, \nu\in\mathbb{Z}}$ so that
$$
f=\sum_{\nu\in\mathbb{Z}}\sum_{Q\in \mathscr{D}_\nu}s_Qa_Q \ \ \text{in $L^2(X)$}
$$
and
\begin{equation}\label{eq1- TL space - prof spm}
\Big\|\Big[\sum_{\nu\in\mathbb{Z}}2^{\nu\alpha q}\Big(\sum_{Q\in \mathscr{D}_\nu}V(Q)^{-1/p}|s_Q|\chi_Q\Big)^q\Big]^{1/q}\Big\|_{p}\lesi \|f\|_{\F^{\alpha,L}_{p,q}}.
\end{equation}

As a consequence, we have
\[
F(\sL)f = \sum_{\nu\in\mathbb{Z}}\sum_{Q\in \mathscr{D}_\nu}s_Q  F(\sL)a_Q .
\]
This, in combination with Lemmas \ref{lem1- thm2 atom Besov} and \eqref{lem2-thm2 atom Besov}, implies
$$
\begin{aligned}
2^{j\alpha}|\psi_j(\sL)F(\sL)f|
&\lesi \sum_{\nu: \nu\geq j}2^{-(\nu-j)(2M-n/r-\alpha)}\mathcal{M}_{r}\Big(\sum_{Q\in \mathscr{D}_\nu}2^{\nu\alpha}|s_Q|V(Q)^{-1/p}\chi_Q\Big)\\
& \ \ \ \ +\sum_{\nu: \nu< j}2^{-(2M-\alpha)(j-\nu)}\mathcal{M}_{r}\Big(\sum_{Q\in \mathscr{D}_\nu}2^{\nu\alpha}|s_Q|V(Q)^{-1/p}\chi_Q\Big)
\end{aligned}
$$
where $\f{n}{s'-\f{\tilde n}{2}}<r<\min\{1, p,q\}$.

It follows that
\begin{equation}
\label{eq- q spm}
\begin{aligned}
\|F(\sL)f\|_{\F^{\alpha,L}_{p,q}(X)}&\sim \Big\|\Big[\sum_{j\in \mathbb{Z}}\left(2^{j\alpha}|\psi_j(\sL)F(\sL)f|\right)^q\Big]^{1/q}\Big\|_p\\
&\lesi \Big\|\Big\{\sum_{j\in\mathbb{Z}}\Big[\sum_{\nu: \nu\geq j}2^{-(\nu-j)(2M-n/r-\alpha)}\mathcal{M}_{r}\Big(\sum_{Q\in \mathscr{D}_\nu}2^{\nu\alpha}|s_Q|V(Q)^{-1/p}\chi_Q\Big)\Big]^q\Big\}^{1/q}\Big\|_p\\
& \ \ \ +\Big\|\Big\{\sum_{j\in\mathbb{Z}}\Big[\sum_{\nu: \nu< j}2^{-(j-\nu)(2M-\alpha)}\mathcal{M}_{r}\Big(\sum_{Q\in \mathscr{D}_\nu}2^{\nu\alpha}|s_Q|V(Q)^{-1/p}\chi_Q\Big)\Big]^q\Big\}^{1/q}\Big\|_p.
\end{aligned}
\end{equation}
If $1\le q<\vc$, then applying Young's inequality we have
\[
\begin{aligned}
\|F(\sL)f\|_{\F^{\alpha,L}_{p,q}(X)}&:=\Big\|\Big[\sum_{j\in \mathbb{Z}}\left(2^{j\alpha}|\psi_j(\sL)F(\sL)f|\right)^q\Big]^{1/q}\Big\|_p\\
&\lesi \Big\| \Big[\sum_{\nu\in \mathbb{Z}} \mathcal{M}_{r}\Big(\sum_{Q\in \mathscr{D}_\nu}2^{\nu\alpha}|s_Q|V(Q)^{-1/p}\chi_Q\Big)^q \Big]^{1/q}\Big\|_p\\
&\lesi \Big\| \Big[\sum_{\nu\in \mathbb{Z}} \Big(\sum_{Q\in \mathscr{D}_\nu}2^{\nu\alpha}|s_Q|V(Q)^{-1/p}\chi_Q\Big)^q \Big]^{1/q}\Big\|_p\\
&\lesi \|f\|_{\F^{\alpha,L}_{p,q}(X)}
\end{aligned}
\]
where in the second inequality we used Fefferman-Stein's inequality and in the last inequality we used \eqref{eq1- TL space - prof spm}.

If $q\in (0,1)$, using the inequality 
\[
\Big(\sum_{j}|a_j|\Big)^q\lesi \sum_j |a_j|^q
\]
for \eqref{eq- q spm}, we obtain

\begin{equation*}
\begin{aligned}
\|F(\sL)f\|_{\F^{\alpha,L}_{p,q}(X)}&\sim \Big\|\Big[\sum_{j\in \mathbb{Z}}\left(2^{j\alpha}|\psi_j(\sL)F(\sL)f|\right)^q\Big]^{1/q}\Big\|_p\\
&\lesi \Big\|\Big\{\sum_{j\in\mathbb{Z}} \sum_{\nu: \nu\geq j}2^{-q(\nu-j)(2M-n/r-\alpha)}\mathcal{M}_{r}\Big(\sum_{Q\in \mathscr{D}_\nu}2^{\nu\alpha}|s_Q|V(Q)^{-1/p}\chi_Q\Big)^q\Big\}^{1/q}\Big\|_p\\
& \ \ \ +\Big\|\Big\{\sum_{j\in\mathbb{Z}}\sum_{\nu: \nu< j}2^{-q(j-\nu)(2M-\alpha)}\mathcal{M}_{r}\Big(\sum_{Q\in \mathscr{D}_\nu}2^{\nu\alpha}|s_Q|V(Q)^{-1/p}\chi_Q\Big)^q\Big\}^{1/q}\Big\|_p.
\end{aligned}
\end{equation*}
Therefore,
\[
\begin{aligned}
\|F(\sL)f\|_{\F^{\alpha,L}_{p,q}(X)}&\lesi \Big\| \Big[\sum_{\nu\in \mathbb{Z}} \mathcal{M}_{r}\Big(\sum_{Q\in \mathscr{D}_\nu}2^{\nu\alpha}|s_Q|V(Q)^{-1/p}\chi_Q\Big)^q \Big]^{1/q}\Big\|_p.
\end{aligned}
\]
At this stage, arguing similarly to the case $1\le q<\vc$, we obtain
\[
\|F(\sL)f\|_{\F^{\alpha,L}_{p,q}(X)}\lesi \|f\|_{\F^{\alpha,L}_{p,q}(X)}.
\]
This completes our proof.
\end{proof}

\medskip

We are ready to give the proof for Theorem \ref{mainthm-spectralmultipliers}.
\begin{proof}[Proof of Theorem \ref{mainthm-spectralmultipliers}:] 
	We note that the condition $s>\f{1}{\tilde q}$ guarantees that $\|F\|_{L^\vc}\lesi \|F\|_{W^{\tilde q}_s}$. As a consequence,
	\[
	\|F\|_{L^\vc} + \sup_{t>0}\|\eta\, \delta_tF\|_{W^{\tilde q}_s}\lesi F(0)+\sup_{t>0}\|\eta\, \delta_tF\|_{W^{\tilde q}_s}.
	\]
Moreover, by using the trick as in the proof of Theorem 5.4 in \cite{KU} without loss of generality we may assume that $L$ is injective which implies $F(0)=0$. For this reason, the right hand sides in \eqref{eq1-thm1} and \eqref{eq2-thm1} becomes $\sup_{t>0}\|\eta\, \delta_tF\|_{W^{\tilde q}_s}$.

Keeping this in mind, we first prove that $F(\sL)$ is bounded on $\F_{p,q}^{\alpha,L}(X)$ for all $1<p,q<\vc$ and $\alpha\in \mathbb{R}$ provided that $s>\f{n}{2}$ and its operator norm is bounded by a multiple of $\sup_{t>0}\|\eta\, \delta_tF\|_{W^{\tilde q}_s}$.
	
	\centerline{\begin{tikzpicture}[scale=2.5]
	\fill[fill=gray!10]
	(0,0) --(0,0.5)--(1.25,1.25)--(1.25,0.5)--(0,0);
	\draw[->] (0,0) -- (1.6,0) node[anchor=north west] (x axis) {$1/p$};
	\draw[->] (0,0) -- (0,1.6) node[anchor=south east] (y axis) {$1/q$};
	\foreach \x/\xtext in {0/O, 0.5/\f{1}{2}, 1,1.25/\f{1}{p_0}} 
	\draw (\x cm,1pt) -- (\x cm,-1pt) node[anchor=north,fill=white] {$\xtext$};
	\foreach \y/\ytext in {0.5/A, 1, 1.25/\f{1}{p_0}} 
	\draw (1pt,\y cm) -- (-1pt,\y cm) node[anchor=east,fill=white] {$\ytext$};
	\draw (0,1) -- (1,1);
	\draw (1,0) -- (1,1);
	\draw (1.25,0) -- (1.25,1.25) node {\hskip 0.4cm $B$};
	\draw (0,1.25) -- (1.25,1.25);
	\draw (0,0)  -- (1.25,1.25);
	\draw (0,0.5) -- (1.25,0.5) node {\hskip 0.4cm $C$};
	\draw (0,0) --(1.25,0.5);
	\draw (0,0.5) --(1.25,1.25);
	\node [below=1cm, align=flush center]
	{
		\hskip 3cm \text{Figure 1}
	};
	\end{tikzpicture}}	
%

Fix $\frac{2n}{2s+n}<p_0<1$ so that $s>n\Big(\f{1}{p_0}-\f{1}{2}\Big)$. Then by (i) of Theorem \ref{thm-spectralmultipliers}, for $s>\f{n}{2}$, $F(\sL)$ is bounded on $\F_{p,2}^{\alpha,L}(X)$ for all $\alpha\in \mathbb{R}$ and $p_0\le p<\vc$ which is corresponding to $\alpha\in \mathbb{R}$ and $(p,q)$ on the interval $AC$ excluding two endpoints. See Figure 1.
 Hence, by the real interpolation result in Theorem \ref{mainthm-Interpolation}, $F(\sL)$ is bounded on $\B_{p,p}^{\alpha,L}(X)\equiv \F_{p,p}^{\alpha,L}(X)$ for all $\alpha\in \mathbb{R}$ and $p_0\le p<\vc$ which is corresponding to $\alpha\in \mathbb{R}$ and $(p,q)$ on the interval $OB$ excluding two endpoints. By using the complex interpolation in Proposition \ref{prop-comple interpolation}, we imply that $F(\sL)$ is bounded on $\F_{p,q}^{\alpha,L}(X)$ for all $\alpha\in \mathbb{R}$ and $(p,q)$ being in the domain bounded by the quadrilateral $OABC$ excluding the intervals $OA$, $AB$ and $OC$. However, if we choose $\frac{2n}{2s+n}<\tilde p<p_0$, repeating the argument above, we can see easily that  $F(\sL)$ is bounded on $\F_{p,q}^{\alpha,L}(X)$ for all $\alpha\in \mathbb{R}$ and $(p,q)$ being in the domain bounded by the quadrilateral $OABC$ excluding the intervals $OA$.

\centerline{\begin{tikzpicture}[scale=4]
	\fill[fill=gray!10]
	(0,0) --(0,0.5)--(1.25,1.25)--(1.25,0.5)--(0,0);
	\draw[->] (0,0) -- (1.6,0) node[anchor=north west] (x axis) {$1/p$};
	\draw[->] (0,0) -- (0,1.6) node[anchor=south east] (y axis) {$1/q$};
	\foreach \x/\xtext in {0/O, 0.5/\f{1}{2}, 1,1.25/\f{1}{p_0}} 
	\draw (\x cm,1pt) -- (\x cm,-1pt) node[anchor=north,fill=white] {$\xtext$};
	\foreach \y/\ytext in {0.5/A, 1, 1.25/\f{1}{p_0}} 
	\draw (1pt,\y cm) -- (-1pt,\y cm) node[anchor=east,fill=white] {$\ytext$};
	\draw (0,1) -- (1,1);
	\draw (1,0) -- (1,1);
	\draw (1.25,0) -- (1.25,1.25) node {\hskip 0.4cm $B$};
	\draw (0,1.25) -- (1.25,1.25);
	\draw (0,0) --(1.25,0.5)  node {\hskip 0.4cm $C$};
	\draw (0,0.5) --(1.25,1.25);
	\draw (0,0) -- (0.8333333333,1) node[below]{\hskip0.3cm $D_0$};
	\draw[->]  (0.8333333333,1) -- (1.208333333285,1.45);
	\draw[->]  (0,0.5) -- (1.041666666,1.25) node[above]{$E_0$};
	\draw[dotted] (0,0) -- (0.6944444444,1) node[above]{$D_1$};
	\draw[dotted,->] (0.6944444444,1) -- (1.0069444444,1.45);
	\draw[dotted] (0,0.5) -- (0.8680555555,1.25) node[above]{{\small $E_1$}};
	\draw (0,1) -- (0.55,1) node[above]{$D_2$};
	\node at (0.6944444444,1) {\textbullet};
	\node at (0.8333333333,1) {\textbullet};
	\node at (0.575,1) {\textbullet};
	\node [below=1cm, align=flush center]
	{
		\hskip 5cm \text{Figure 2}
	};
	\end{tikzpicture}	}

We next prove that for each $(p,q)$ on the ray $\overrightarrow{OD_0}$, $F(\sL)$ is bounded on $\F_{p,q}^{\alpha,L}(X)$ for all $\alpha\in \mathbb{R}$ if $s>n(\f{1}{p\wedge q}-\f{1}{2})=n(\f{1}{q}-\f{1}{2})$. See Figure 2. Clearly, it suffices to verify this assertion for $(p,q)\in \overrightarrow{OD_0}$ with $0<p<1$. Indeed, for any $\theta\in (0,1)$ and $(p_1,q_1), (p_2,q_2)\in \overrightarrow{OD_0}$ so that $q_1>1$, $0<q_2<1$ and
\[
\f{1}{p}=\f{ \theta}{p_1} +\f{1-\theta}{p_2}, \ \ \ \ \f{1}{q}=\f{\theta}{q_1} +\f{1-\theta}{q_2},
\]
we have $p_1>q_1$ and $p_2>q_2$. Since $(p_1,q_1)$ belongs to the interior of the quadrilateral $OABC$, we have
\begin{equation}\label{eq1-interpolation proof main result}
\|F(\sqrt L)\|_{\F^{\alpha, L}_{p_1,q_1}(X)\to \F^{\alpha, L}_{p_1,q_1}(X)}\lesi  \sup_{t>0}\|\eta\, \delta_tF\|_{W^{\tilde q}_{\f{n}{2}+\epsilon}}
\end{equation}
where $\epsilon>0$ will be fixed later.

Moreover, by (ii) of Theorem \ref{thm-spectralmultipliers},
\begin{equation}\label{eq2-interpolation proof main result}
\|F(\sqrt L)\|_{\F^{\alpha, L}_{p_2,q_2}(X)\to \F^{\alpha, L}_{p_2,q_2}(X)}\lesi  \sup_{t>0}\|\eta\, \delta_tF\|_{W^{\tilde q}_{\f{n}{q_2}+\f{\tilde n}{2}+\epsilon}}
\end{equation}
The estimates \eqref{eq1-interpolation proof main result} and \eqref{eq2-interpolation proof main result}, along with the complex interpolation result in Proposition \ref{prop-comple interpolation}, imply that
 \begin{equation}\label{eq3-interpolation proof main result}
 \begin{aligned}
  \|F(\sqrt L)\|_{\F^{\alpha, L}_{p,q}(X)\to \F^{\alpha, L}_{p,q}(X)}&\sim\|F(\sqrt L)\|_{\left(\F^{\alpha, L}_{p_1,q_1}(X),\F^{\alpha, L}_{p_2,q_2}(X)\right)_\theta\to \left(\F^{\alpha, L}_{p_1,q_1}(X),\F^{\alpha, L}_{p_2,q_2}(X)\right)_\theta}\\
  &\lesi  \sup_{t>0}\|\eta\, \delta_tF\|_{\Big(W^{\tilde q}_{\f{n}{2}+\epsilon},W^{\tilde q}_{\f{n}{q_2}+\f{\tilde n}{2}+\epsilon}\Big)_\theta}\\
  &\sim  \sup_{t>0}\|\eta\, \delta_tF\|_{\Big(W^{\tilde q}_{\theta(\f{n}{2}+\epsilon)+(1-\theta)(\f{n}{q_2}+\f{\tilde n}{2}+\epsilon)}\Big)}.
 \end{aligned}
 \end{equation}
Note that 
\[
\begin{aligned}
\theta\Big(\f{n}{2}+\epsilon\Big)+(1-\theta)\Big(\f{n}{q_2}+\f{\tilde n}{2}+\epsilon\Big)&=\f{\theta n}{2}+\f{n(1-\theta)}{q_2}+\f{(1-\theta)\tilde n}{2} +\epsilon\\
&=\f{\theta n}{2}+\f{n}{q}-\f{(\theta n}{q_1}+\f{(1-\theta)\tilde n}{2} +\epsilon\\
&=n\Big(\f{1}{q}-\f{1}{2}\Big)-\f{\theta n}{q_1}+\f{(1+\theta)n}{2}+\f{(1-\theta)\tilde n}{2} +\epsilon
\end{aligned}
\]
Hence, for any $\tilde \epsilon>\epsilon$, we are able to choose $\theta \uparrow 1$ and $q_1\downarrow 1$  so that 
\[
n\Big(\f{1}{q}-\f{1}{2}\Big)-\f{\theta n}{q_1}+\f{(1+\theta)n}{2}+\f{(1-\theta)\tilde n}{2} +\epsilon<n\Big(\f{1}{q}-\f{1}{2}\Big) +\tilde \epsilon.
\]
This, in combination with \eqref{eq3-interpolation proof main result}, implies that
\begin{equation}\label{OD}
\|F(\sqrt L)\|_{\F^{\alpha, L}_{p,q}(X)\to \F^{\alpha, L}_{p,q}(X)}\lesi  \sup_{t>0}\|\eta\, \delta_tF\|_{W^{\tilde q}_{s}}
\end{equation}
as long as $s>n(\f{1}{q}-\f{1}{2})$, $(p,q)\in \overrightarrow{OD_0}$ and $\alpha\in \mathbb{R}$.

Therefore, using the complex interpolation theorem again, we obtain that \eqref{OD} holds true for all  $(p,q)$ being on the ray $\overrightarrow{OD_0}$ and in interior of the convex hull of the point $A$ and the ray $\overrightarrow{OD_0}$. As a consequence, \eqref{OD} holds true for all  $(p,q)\in \overrightarrow{OD_0} \cup AE_0\backslash\{A\}$. Since $\frac{2n}{2s+n}<p_0<1$, it follows that 
\begin{equation}\label{ODs}
\|F(\sqrt L)\|_{\F^{\alpha, L}_{p,q}(X)\to \F^{\alpha, L}_{p,q}(X)}\lesi  \sup_{t>0}\|\eta\, \delta_tF\|_{W^{\tilde q}_{s}}
\end{equation}
as long as $s>\f{n}{2}$, $(p,q)\in OD_0\cup AE_0\backslash\{A,O\}$ and $\alpha\in \mathbb{R}$.

Repeating this argument, we can find a sequence of points $\{D_k\}_{k=0}^\vc$ so that \eqref{ODs} holds true whenever $(p,q)\equiv D_k, k=1,1,2,\ldots$. Moreover, by a straightforward calculations we can show that $D_k \to (0,1)$ as $k\to \vc$. Using the complex interpolation result in Proposition \ref{prop-comple interpolation} for the interval $OB\backslash O$ and the sequence $\{D_k\}_{k=0}^\vc$, we derive that 
\begin{equation*}
\|F(\sqrt L)\|_{\F^{\alpha, L}_{p,q}(X)\to \F^{\alpha, L}_{p,q}(X)}\lesi  \sup_{t>0}\|\eta\, \delta_tF\|_{W^{\tilde q}_{s}}
\end{equation*}
as long as $s>\f{n}{2}$, $p\ge q>1$ and $\alpha\in \mathbb{R}$.

Applying the duality in Proposition \ref{prop-duality},
 \begin{equation}\label{all p q}
 \|F(\sqrt L)\|_{\F^{\alpha, L}_{p,q}(X)\to \F^{\alpha, L}_{p,q}(X)}\lesi  \sup_{t>0}\|\eta\, \delta_tF\|_{W^{\tilde q}_{s}}
 \end{equation}
 as long as $s>\f{n}{2}$, $1<p,q<\vc$ and $\alpha\in \mathbb{R}$.
 
 At this stage, with \eqref{all p q} on hands, using the similar argument  to  the proof of \eqref{OD}, we obtain that
  \begin{equation}\label{all p q}
  \|F(\sqrt L)\|_{\F^{\alpha, L}_{p,q}(X)\to \F^{\alpha, L}_{p,q}(X)}\lesi  \sup_{t>0}\|\eta\, \delta_tF\|_{W^{\tilde q}_{s}}
  \end{equation}
  as long as $s>n(\f{1}{1\wedge p\wedge q}-\f{1}{2})$, $0<p,q<\vc$ and $\alpha\in \mathbb{R}$.
  
  This completes the proof of (i).
  
  \medskip
  
  \noindent (ii) The proof of the item (ii) follows directly from (i) and the real interpolation result in Theorem \ref{mainthm-Interpolation}.
  
  This completes our proof.
\end{proof}

{\bf Acknowledgement.}  Xuan Thinh Duong was supported by Australian Research Council through the ARC grant DP160100153.


\begin{thebibliography}{999}

\bibitem{Alex} G. Alexopoulos, Spectral multipliers on Lie groups of polynomial growth, Proc. Amer. Math. Soc. 46 (1994), 457--468.

\bibitem{Amenta} A. Amenta, Interpolation and embeddings of weighted tent spaces, to appear in Journal of Fourier Analysis and Applications. Available at https://arxiv.org/abs/1509.05699.

	
	
	
	
	\bibitem{BL} J. Bergh and J. L\"{o}fstr\"{o}m, Interpolation spaces, an introduction, Springer-Verlag, 1976.
	
	
	
	
	
	\bibitem{BBD} H.-Q. Bui, T. A. Bui and X. T. Duong, Weighted Besov and Triebel--Lizorkin spaces associated to operators, Available at: https://arxiv.org/pdf/1809.02795.pdf
	
	
	
	
	 
	 
	  
	  
	   
	 
	 \bibitem{B} T. A. Bui, Notes on boundedness of spectral multipliers on Hardy spaces associated to operators, Nagoya Math. J. 203 (2011), 109--122. 
	 
	 
	 \bibitem{BDK} T. A. Bui, X. T Duong and F. K. Ly, Maximal function characterizations for new local Hardy type spaces on spaces of homogeneous type, Trans. Amer. Math. Soc. 370 (2018), no. 10, 7229--729.
 	
	
	
	
	
	\bibitem{CDa} F. Cacciafesta and P. D'Ancona, Weighted $L^p$ estimates for powers of self-adjoint operators, Adv. Math. 229 (2012), no. 1, 501--530.
	
	
	\bibitem{CF} S.-Y. Chang and R. Fefferman, A continuous version of duality of $H^1$ and BMO on the bidisc, Ann. of Math. 112 (1980), 179--201.
	\bibitem{C} M. Christ, A $Tb$ theorem with remarks on analytic capacity and the Cauchy integral, Colloq. Math., 61 (1990), 601--628.
	
	\bibitem{C2} M. Christ, $L^p$ bounds for spectral multipliers on nilpotent groups,
	Trans. Amer. Math. Soc. 328 (1991), 73--81.
	
	\bibitem{CW} R. R. Coifman and G. Weiss, Analyse harmonique non-commutative sur certains espaces homogenes, Lecture Notes in Mathematics no. 242, Springer-Verlag, 1971.
	
	
	
	
	
	
	
	
	\bibitem{DOS}  X. T. Duong, E.M. Ouhabaz  and A. Sikora,
	Plancherel--type estimates and sharp spectral multipliers.
	J. Funct. Anal. {196} (2002),  443--485.
	
	
	
	\bibitem{DSY} X.T. Duong, A. Sikora and L. Yan, Weighted norm inequalities, Gaussian bounds and sharp spectral multipliers, J. Funct. Anal.  260 (2011),  1106--1131.
	
	\bibitem{DY} X.T. Duong and L. Yan, Spectral multipliers for Hardy spaces associated to non-negative self-adjoint operators satisfying Davies--Gaffney estimates, J. Math. Soc. Japan 63 (2011), no. 1, 295--319.
	
	\bibitem{DY1} X.T. Duong and L.X. Yan, New function spaces of BMO type, the
	John--Nirenberg inequality, interpolation and applications, Comm.
	Pure Appl. Math. 58 (2005), 1375--420.
	
	
	
	
	\bibitem{FJ2} M. Frazier and B. Jawerth,  A discrete transform and decomposition of distribution spaces, J. Funct. Anal. 93 (1990), 34--170.
	
	\bibitem{G.etal} A. G. Georgiadis and G. Kerkyacharian, G. Kyriazis, and P. Petrushev, Homogeneous Besov and Triebel--Lizorkin spaces associated to non--negative self--adjoint operators. J. Math. Anal. Appl. 449 (2017), no. 2, 1382--1412.
	 
	 \bibitem{G.etal2} A. G. Georgiadis, G. Kerkyacharian, G. Kyriazis and P. Petrushev, Atomic and molecular decomposition of homogeneous spaces of distributions associated to non-negative self-adjoint operators. Available at: https://arxiv.org/abs/1805.01444.
	
	 
	  
	\bibitem{GLY} L. Grafakos, L. Liu and D. Yang, Vector-valued singular integrals and maximal functions on spaces of homogeneous type, Mathematica Scandinavica, 104 (2009), 296--310.
	
	 
	
	
	 \bibitem{HMY} Y. S. Han, D. M\"uller and D. Yang, A theory of Besov and Triebel--Lizorkin spaces on metric measure spaces modeled on Carnot--Carathéodory spaces. Abstr. Appl. Anal. 2008, Art. ID 893409, 250 pp. 
	 
	 \bibitem{HPW} Y. S. Han, M. Paluszy\'nski and G. Weiss, A new atomic decomposition for the Triebel-Lizorkin spaces, Harmonic analysis and operator theory (Caracas, 1994), 235–249, Contemp. Math., 189, Amer. Math. Soc., Providence, RI, 1995. 
	  
	\bibitem{HS} Y. S. Han and E. T. Sawyer, Littlewood-Paley theory on spaces of homogeneous type and the classical function spaces, Mem. Amer. Math. Soc. 110 (1994), no. 530, vi+126 pp.
	
	\bibitem{Heb} W. Hebisch, A multiplier theorem for Schr ̈odinger operators,
	Colloq. Math. 60/61 (1990), 659--664.
	
	\bibitem{Ho} L. H\"ormander, Estimates for translation invariant operators in $L^p$ spaces, Acta Math. 104 (1960), 93--140.
	
	\bibitem{HLMMY}S. Hofmann, G. Lu, D. Mitrea, M. Mitrea and L. Yan, Hardy spaces associated to non-negative self-adjoint operators satisfying Davies-Gaffney estimates, Mem. Amer. Math. Soc. 214 (1007) (2011).
	
	\bibitem{HHM} S. Hofmann, S. Mayboroda and A. McIntosh, Second  order  elliptic  operators  with  complex bounded measurable coefficients in $L^p$, Sobolev and Hardy spaces, Ann. Sci. Ec. Norm. Sup\'er. (4) 44 (2011), no. 5, 723--800.
	
	\bibitem{JY} R. Jiang and D. Yang, Orlicz--Hardy spaces associated with operators satisfying Davies--Gaffney estimates, Commun. Contemp. Math. 13 (2011), no. 2, 331--373.
	 
	
	\bibitem{KP} G. Keryacharian and P. Petrushev, Heat kernel based decomposition of spaces of distributions in the framework of Dirichlet spaces,  Trans. Amer. Math. Soc. 367 (2015), no. 1, 121--189.
	
	\bibitem{KU} P. C. Kunstmann and M. Uhl, Spectral multiplier theorems of H\"ormander type on Hardy and lebesgue spaces, Journal of Operator Theory 73 (2015), 27--69.
	
	
	\bibitem{KW} D.S. Kurtz and R.L. Wheeden, Results on weighted norm inequalities for multipliers, Trans. Amer. Math. Soc. 255 (1979), 343--362.
	
	
	
	\bibitem{O} E. M. Ouhabaz, Analysis of heat equations on domains, London Math. Soc. Monogr., 31, Princerton Univ. Press, 2005.
	
	\bibitem{MM} G. Mauceri and S. Meda, Vector--valued multipliers on stratified
	groups, Rev. Mat. Iberoamericana 6 	(1990), 141--154.
	
	\bibitem{MS} D. M\"uller and E.M. Stein, On spectral multipliers for Heisenberg and related groups, J. Math. Pures Appl. 73 (1994), 413--440.
	
	
	 \bibitem{PX} P. Petrushev and Y. Xu, Decomposition of spaces of distributions induced by Hermite expansions, J. Fourier Anal. Appl. 14 (2008), no. 3, 372--414.
	 
	 
	 
	
	
	
	
	

     \bibitem{Tri} H. Triebel, Complex interpolation and Fourier multipliers for the spaces $B^s_{p,q}$ and $F^s_{p,q}$ of Besov--Hardy--Sobolev type: the case $0<p\le \vc, 0<q\le \vc$, Math. Z. 176 (1981), 495--510.  

	\bibitem{Tr} H. Triebel, Theory of Function Spaces, Monogr. Math., vol.78, Birkh\"auser, Basel, 1983.
	
	


\end{thebibliography}
\end{document}